\documentclass[11pt,twoside,reqno]{amsart}
\usepackage{amsmath,amssymb,amscd}

\usepackage{bbm}
\usepackage{verbatim}
\usepackage{graphicx,xcolor}
\usepackage{enumerate}
\usepackage[all]{xy}
\usepackage[hyperindex=true,plainpages=false,colorlinks=false,pdfpagelabels]{hyperref}
\usepackage{pdfpages}

\newtheorem{theorem}{Theorem}[section]
\newtheorem{proposition}[theorem]{Proposition}
\newtheorem{lemma}[theorem]{Lemma}
\newtheorem{corollary}[theorem]{Corollary}
\newtheorem{theo}{Theorem}
\newtheorem{cor}{Corollary}

\theoremstyle{definition}
\newtheorem{definition}[theorem]{Definition}
\newtheorem{remark}[theorem]{Remark}
\newtheorem{assumption}[theorem]{Assumption}

\newcommand{\nbiga}{\mathcal{A}}
\newcommand{\nbigb}{\mathcal{B}}
\newcommand{\nbigc}{\mathcal{C}}
\newcommand{\nbigd}{\mathcal{D}}
\newcommand{\nbigp}{\mathcal{P}}

\newcommand{\nbigctilde}{\widetilde{\nbigc}}
\newcommand{\alphatilde}{\widetilde{\alpha}}
\newcommand{\betatilde}{\widetilde{\beta}}
\newcommand{\gammatilde}{\widetilde{\gamma}}

\newcommand{\delbar}{\overline{\del}}

\newcommand{\cnum}{{\mathbb C}}
\newcommand{\vecv}{{\boldsymbol v}}
\newcommand{\vecu}{{\boldsymbol u}}

\newcommand{\lrarr}{\longrightarrow}

\def\tr{\mathop{\rm tr}\nolimits}
\def\Hom{\mathop{\rm Hom}\nolimits}
\def\Re{\mathop{\rm Re}\nolimits}

\numberwithin{equation}{section}
\newcommand{\wt}[1]{\widetilde{#1}}

\DeclareMathOperator{\SL}{SL}

\DeclareMathOperator{\Tr}{tr}

\newcommand{\dvector}[1]{{\left(\begin{matrix}#1\end{matrix}\right)}}

\DeclareMathOperator{\del}{\partial}

\newcommand{\R}{\mathbb{R}}
\newcommand{\C}{\mathbb{C}}
\newcommand{\CP}{\C P}
\newcommand{\N}{\mathbb{N}}
\newcommand{\Z}{\mathbb{Z}}
\begin{document}

\title[Holomorphic $\mathfrak{sl}(2,\C)$-systems with Fuchsian monodromy]{Holomorphic $\mathfrak{sl}(2,\C)$-systems with Fuchsian monodromy (with an appendix by Takuro Mochizuki)}

\author[I. Biswas]{Indranil Biswas}

\address{School of Mathematics, Tata Institute of Fundamental Research,
Homi Bhabha Road, Mumbai 400005, India}

\email{indranil@math.tifr.res.in}

\author[S. Dumitrescu]{Sorin Dumitrescu}

\address{Universit\'e C\^ote d'Azur, CNRS, LJAD, France}

\email{dumitres@unice.fr}

\author[L. Heller]{Lynn Heller}
\address{Institute of Differential Geometry,
Leibniz Universit\"at Hannover,
Welfengarten 1, 30167 Hannover}
\email{lynn.heller@math.uni-hannover.de}

\author[S. Heller]{Sebastian Heller}

\address{Institute of Differential Geometry,
Leibniz Universit\"at Hannover,
Welfengarten 1, 30167 Hannover}

\email{seb.heller@gmail.com}

\address{Research Institute for Mathematical Sciences,
Kyoto University, Kyoto 606-8502, Japan}

\email{takuro@kurims.kyoto-u.ac.jp}



\date{\today}

\begin{abstract}
For every integer $g \,\geq\, 2$ we show the existence of a compact Riemann surface $\Sigma$ of genus $g$
such that the rank two trivial holomorphic vector bundle ${\mathcal O}^{\oplus 2}_{\Sigma}$ admits
holomorphic connections with  $\text{SL}(2,{\mathbb R})$  monodromy and maximal Euler class. Such a monodromy representation
is known to coincide with the Fuchsian uniformizing representation for some Riemann surface of genus $g$. The construction carries over to all very stable and compatible real holomorphic structures for the topologically trivial rank two bundle over $\Sigma$ and gives the existence of  holomorphic connections with Fuchsian monodromy in these cases as well.
\end{abstract}
\maketitle

\tableofcontents

\section*{Introduction}\label{se1}

For a compact  Riemann surface $\Sigma$ the holomorphic Riemann-Hilbert correspondence associates to every pair 
$(V, \, \nabla)$, consisting of a (flat) holomorphic connection $\nabla$  on a holomorphic vector bundle $V$ over $\Sigma$, 
its monodromy homomorphism. This is an equivalence of categories (see for instance \cite{De} or 
\cite[p. 544]{Ka}). For surfaces with nonabelian fundamental group finding holomorphic connections with prescribed monodromy
behavior is notoriously 
difficult and an obstacle to a deeper understanding of various mathematical problems
ranging from 
algebraic geometry and number theory, over
geometric structures on manifolds  \cite{AQ, DM, Th},  to constructions
in quantum field theories and mirror symmetry \cite{AGM,W,GMN,FGuTe} and to the  theory of  harmonic maps and minimal surfaces \cite{hitchin, Wo, HHSch,Tr}.

In this paper we restrict to the case of $\text{SL}(2, {\mathbb C})$--connections over compact Riemann surfaces of
genus $g\geq2$. This case is of particular interest as it is deeply linked to the geometry of the underlying
 surface. Starting from the XIXth century mathematicians have investigated group representations appearing as 
monodromy of solutions to algebraic differential equations on the complex domain.  
The relationship to geometry stems from the fact that the inverse of solutions to certain linear differential equations parametrize the Riemann surface. As discovered by Poincar\'e and Klein 
(see \cite{StG} for a historical survey of the subject), every Riemann surface  can be realized as a quotient of the hyperbolic plane $\mathcal H^2$ by a Fuchsian group 
(a torsion-free, discrete, and cocompact subgroup of $\text{PSL}(2, \R)$) identifying 
  the space of Fuchsian representations with the Teichm\"uller space. Lifting Fuchsian representations from PSL$(2, \R)$ to SL$(2, \R)$,  they 
can be considered as   monodromy representations of holomorphic
 SL$(2, \C)$-connections on a fixed Riemann surface $\Sigma$ via Riemann-Hilbert correspondence.
  The holomorphic structure on the rank two vector bundle given by the uniformization of $\Sigma$ is the unique nontrivial extension of $L^{-1}$ by $L$, where $L$ is a theta characteristic on
$\Sigma$  \cite{Gu}. This bundle will be referred to as the uniformization bundle.
Note that Fuchsian representations  are    SL$(2, \R)$-representations with maximal Euler class  $g-1$. 
 This gives $2^{2g}$ connected components with maximal Euler class in the space of SL$(2, \R)$-representations 
corresponding to the different choices of the  theta characteristic \cite{Mi,Hi1,Gold88}.

In this context it is natural  to ask which holomorphic rank two bundles over a given Riemann surface $\Sigma$ admit holomorphic connections with Fuchsian monodromy representations. Indeed, this question was first raised by Katz in \cite[p.~555--556]{Ka} (where the question is attributed to Bers) in 1978 and is still unsolved. 
On the other hand, the analogue to Bers' question for the compact group ${\rm SU}(2)$ is fully understood. 
The celebrated 
Narasimhan-Seshadri Theorem shows that every stable holomorphic structure admits a unique compatible flat 
connection with irreducible unitary monodromy and vice versa.
Motivated by problems in algebraic geometry and number theory, e.g., Weil conjecture,
a related question 
of realizing Fuchsian representations as the monodromy homomorphism of  regular 
singular $\text{SL}(2, {\mathbb C})$-connections on the uniformization bundle over (marked) Riemann surfaces was addressed by Faltings 
\cite{Fa}. 
Remarkably, even when restricting to the trivial rank two holomorphic bundle, it was previously unknown  whether a holomorphic  connection $\nabla$   with Fuchsian monodromy representation exists. 
 This is the main question to be addressed in the present article. We prove

\begin{theo}[Main Theorem]\label{thi}
For  every  $k\,\in\,\N^{\geq3}$ there exists a (hyperelliptic) Riemann surface $\Sigma_k$ of genus $k-1$ such that
the trivial holomorphic rank two  bundle 
admits infinitely many holomorphic connections
with Fuchsian monodromy representation.
\end{theo}

 A major difference to Narasimhan-Seshadri Theorem when considering the split real group 
${\rm SL}(2,\R)$  is that uniqueness fails, e.g., our Main Theorem 
shows the existence of infinitely many holomorphic 
connections with Fuchsian monodromy on the trivial holomorphic bundle. Likewise, for the holomorphic structure given by the 
uniformization bundle, the infinitely many holomorphic connections with Fuchsian monodromy correspond to integral graftings, see  \cite{Mas,Hej, SuT,Fa,Gold87}.
Although other holomorphic bundles with holomorphic connections with Fuchsian monodromy do exist, no explicit example other than 
the uniformization bundle itself were found.

Our Main Theorem \ref{thi} is in fact a consequence of an additional real symmetry of the considered Riemann surface $\Sigma_k$. Therefore, the proof carries over verbatim to  all very stable holomorphic structures -- i.e., their (non-zero) Higgs fields are not nilpotent-- on the topologically trivial rank two bundle compatible with the construction and with the real symmetry of the Riemann surface (as specified in Lemma \ref{tausymcon}). The space of these real holomorphic structures can be identified with 
a circle with a single point removed in a projective line.
An immediate corollary is

\begin{cor}\label{cor1}
For  every  $k\,\in\,\N^{\geq3}$ there exists a (hyperelliptic) Riemann surface $\Sigma_k$ of genus $k-1$ such that
all very stable and compatible real holomorphic structures of the topologically trivial rank two bundle over $\Sigma_k$ 
admit  infinitely many holomorphic connections
with Fuchsian monodromy representation.
\end{cor}

In a similar vein, Ghys raised the question about whether there is a pair $(\Sigma,\, 
\nabla)$ consisting of a compact Riemann surface of genus $g \,\geq\, 2$ and an irreducible holomorphic 
connection $\nabla$ on the rank two trivial holomorphic vector bundle 
such that the image of the monodromy homomorphism of $\nabla$ lies in a cocompact lattice 
of $\text{SL}(2, {\mathbb C})$. Such a pair would give rise to a nontrivial holomorphic map from the Riemann 
surface $\Sigma$ to the compact quotient of $\text{SL}(2, {\mathbb C})$ by that cocompact lattice. Constructing 
such holomorphic maps is also known as the Margulis problem (see \cite{CDHL} for the discussion about Ghys question and 
Margulis problem).

Motivated by the above question of Ghys, the authors of \cite{CDHL} initiated  a study of the Riemann-Hilbert 
correspondence for genus two surfaces and $\text{SL}(2, {\mathbb C})$--connections. Their main result  asserts that the 
Riemann-Hilbert monodromy mapping,  which associates to an irreducible holomorphic differential system its monodromy representation, 
is a local biholomorphism.
Then Theorem \ref{thi} and the result of \cite{CDHL} gives

\begin{cor}\label{cori}
There exists a nonempty open subset $\mathcal U$ of the Teichm\"uller space of
compact curves of genus $g\,=\,2$ such that every $\Sigma
\,\in\, {\mathcal U}$ possesses a holomorphic connection $\nabla({\Sigma})$ on ${\mathcal O}^{\oplus 2}_{\Sigma}$
with quasi-Fuchsian \footnote{A representation of a surface group is called quasi-Fuchsian,
if the  monodromy homomorphism is faithful and has discrete image in ${\rm SL}(2, {\mathbb C})$ admitting
a Jordan curve as limit set for its action on $\CP^1$. }
 monodromy representation.

Every curve $\Sigma\,\in\, \mathcal U$ therefore admits a nontrivial holomorphic map into the quotient of ${\rm SL}(2, {\mathbb 
C})$ by a quasi-Fuchsian group as the image of the monodromy homomorphism of $\nabla({\Sigma})$.
\end{cor}

Theorem \ref{thi} and Corollary  \ref{cori} are geometrization results through 
holomorphic $\text{SL}(2, {\mathbb C})$--connections on the trivial bundle instead of the usual hyperbolic or 
Bers simultaneous uniformization for quasi-Fuchsian representations. It should be mentioned that, in higher 
Teichm\"uller spaces, geometrizations results for representations of fundamental group of surfaces into Lie groups 
is currently a very lively and dynamic field of research (see for instance \cite{BIW,GW,La} and references therein).

\section*{Strategy}
We show the existence of holomorphic connections with Fuchsian monodromy representation for particular hyperelliptic surfaces $\Sigma_k$ of genus $(k-1)$ given by a   totally branched   $k$-fold covering $f_k$ of ${\mathcal S}_4$ -- the complex projective line with four 
marked points $\pm1,\, \pm \sqrt{-1}$. On $\Sigma_k$  there are two connections of particular interest; the trivial de Rham differential $d$ and the uniformizing connection $\nabla^U$ of $\Sigma_k$. Both connections can be realized, modulo singular gauge transformations, as the pull-back  of 
the logarithmic connections $D$ (Proposition 
\ref{red}) and $\wt\nabla$ (Proposition \ref{FRS})  by $f_k$. 
Our aim is to deform $D$ by a parabolic Higgs field such that the new connection has real monodromy, lies in the connected component of $\wt \nabla$ and pulls back to $\Sigma_k$ as a holomorphic connection (without singularities). 
 
The moduli space of logarithmic connections on $\mathcal S_4$ has a natural set 
of coordinates given by the abelianization procedure \cite{HeHe}. These coordinates determine logarithmic 
connections on $\mathcal S_4$ as a twisted push forward of flat line bundle connections on the torus $\Sigma_2$ obtained by the 
branched double cover $f_2$ of $\mathcal S_4$. The twist is given by some 
meromorphic off-diagonal 1-forms determined by the flat line bundle and the eigenvalues of the 
residues. We restrict to the most symmetric case, where the behavior of the logarithmic connection at every 
marked point of $\Sigma_2$ is the same. More precisely, we consider connections on the torus $\Sigma_2$ with four marked points 
that descend to connections on the torus $T^2$ with only one marked point by taking the quotient with respect to its
half-lattice. In this way Theorem \ref{thm:4:1} identifies the moduli space of logarithmic connections on $T^2$ with the moduli space of logarithmic connections on $\mathcal S_4$. Moreover, $D$ 
is identified with a  connection $\wt D$ on the torus $T^2$ with one marked point in Lemma \ref{abelred}.

The crucial idea is to consider the asymptotic behavior of the family of connections 
$$ \wt D + t \wt\Phi\,,$$
where $t \,\in \,\R$ and $\wt\Phi$ is a specific parabolic Higgs field of $\wt D.$
By Theorem \ref{thm:4:1}, this family
corresponds to $\nabla^t = D+t\Phi$ on $\mathcal S_4,$ where $\Phi$ is the corresponding parabolic Higgs field of $D$.
By construction all connections $f_k^*\nabla^t$ have the same underlying holomorphic structure, namely the trivial one induced by the de Rham differential $d.$ For $t$ large we then use 
WKB analysis and an additional real involution of the torus (Lemma \ref{tausymcon2})
to ensure the existence of 
a sequence $(t_n)_{n\in \N} \,\subset\, \R$ such that $\nabla^{t_n}$ has real monodromy (Corollary \ref{tn}). 
The necessary WKB analysis result is proved by Takuro Mochizuki in the Appendix.

Since the pull-back under $f_k$ preserves the connected components of real representations, it remains to show that $\nabla^{t_n}$ lies in the same connected component as $\wt\nabla$ 
on $\mathcal S_4$. To do so, we compute that $\wt \nabla$ is also induced by 
a singular connection $\nabla^F$ on the one-punctured torus in Lemma \ref{lem:FFuchs}. The claim then follows 
from the fact that the four components of logarithmic connections with SL$(2,\R)$-monodromy on the one-punctured torus 
are mapped into the same real component of the moduli space on $\mathcal S_4$ via Theorem 
\ref{thm:4:1}. Therefore, the pull-back $f_k^*\nabla^{t_n}$ to $\Sigma_k$ is Fuchsian and has trivial holomorphic 
structure.

In fact, it is necessary to consider singular connections 
on the one-punctured torus, since there exists 3 other components of irreducible 
SL$(2,\R)$-representations on the four-punctured sphere, whose boundary contain reducible connections and do not lift to the Fuchsian 
component on $\Sigma_k$.
Related examples of irreducible holomorphic $\text{SL}(2, {\mathbb 
C})$--connections with real monodromy on the trivial holomorphic rank two bundle
were constructed in \cite{BDH}. However, these connections  are never of 
maximal Euler class. 

\section{Preliminaries: Logarithmic connections and parabolic bundles}\label{ss:parabolic}

Let $\Sigma$ be a compact connected Riemann surface; its holomorphic cotangent bundle is
denoted by $K_\Sigma$. An $\rm{SL}(2, \C)$--bundle on $\Sigma$ is a holomorphic rank two
vector bundle $V$ over $\Sigma$ with trivial determinant, i.e., the
line bundle $\det V\,=\, \bigwedge^2 V$ is holomorphically trivial.

Let ${\mathbf D}\,=\, p_1+\ldots +p_n$ be a divisor on $\Sigma$ with pairwise distinct points
$p_i \,\in\, \Sigma$. Consider a holomorphic $\rm{SL}(2, \C)$--bundle $V$ on $\Sigma$ together with its sheaf $\mathcal V$ of holomorphic sections and its Dolbeault operator  $\bar\partial.$
A
\textit{logarithmic} $\rm{SL}(2, \C)$--connection  $\nabla=\bar\partial+\partial^\nabla$ on $V$ with polar part
contained in $\mathbf D$ is given by a holomorphic differential operator
$$
\partial^\nabla\, :\, \mathcal V\, \longrightarrow\, \mathcal V\otimes K_\Sigma\otimes {\mathcal O}_\Sigma({\mathbf D})
$$
satisfying the Leibniz rule 
$$ \partial^\nabla(fs)\,=\, f\nabla(s)+ s\otimes df\,  $$
for all locally defined holomorphic sections $s$ of $V$ and locally defined holomorphic
functions $f$  on $\Sigma$, such that
the induced differential operator on $\det V$ coincides with the de Rham differential $d$ on ${\mathcal O}_\Sigma$.

Since $\Sigma$ is of complex dimension one, all logarithmic connections over $\Sigma$ are flat.
Moreover,  at every singular point $p_j$, $1\, \leq\, j\, \leq\, n$, of a logarithmic $\rm{SL}(2, \C)$--connection $\nabla$
on $V$ 
the residue \[\text{Res}_{p_j}(\nabla)\,\in\,\text{End}(V_{p_j})\]
is tracefree.

If the two eigenvalues $\lambda_{j,1},\, \lambda_{j,2}$ of the residue $\text{Res}_{p_j}(\nabla)$ do not differ
by an integer (this is known as the non-resonancy condition), then the local monodromy of $\nabla$ around $p_j$ 
is conjugate to the diagonal matrix with entries $\exp(-2\pi\sqrt{-1}\lambda_{j,1})$ and
$\exp(-2\pi\sqrt{-1}\lambda_{j,2})$ (see \cite[p.~53, Th\'eor\`eme 1.17]{De}).
If $\frac{1}{n_j}$ is an eigenvalue of the residue, with $n_j \,\geq\, 2$ an integer, the local
monodromy of $\nabla$ at $p_j$ is a rational rotation on the eigenlines.

Let $V$ be a holomorphic $\rm{SL}(2, \C)$--bundle on $\Sigma$. A parabolic structure $\mathcal P$ on $V$ with parabolic
divisor ${\mathbf D}\,=\, p_1+\ldots +p_n$ consists of quasiparabolic lines $L_j\, \subset\, V_{p_j}$
together with weights $\rho_j\, \in\, ]0,\, \tfrac{1}{2}[$ for
every $1\, \leq\, j\, \leq\, n$. For a holomorphic line subbundle $W\, \subset\, V$ the parabolic degree is given by
\[\text{par-deg}(W)\,:=\, {\rm degree}(W)+\sum_{j=1}^n \rho^W_j\, ,\]
where $\rho^W_j\,= \, \rho_j$ if $W_{p_j}\,=\,L_j$ and $\rho^W_j\,=\,
-\rho_j$ if $W_{p_j}\,\neq\, L_j$; see \cite{MS}, \cite{MY}.

\begin{definition}
A parabolic bundle $(V,\,\mathcal P)$ is called {\it stable} (respectively, {\it semistable}) if
$\text{par-deg}(W)\, <\, 0$ (respectively, $\text{par-deg}(W)\, \leq\, 0$)
for every holomorphic line subbundle $W\,\subset \, V$. A parabolic bundle will be called \textit{unstable} if it is not semistable.
\end{definition}

Take a non-resonant logarithmic $\rm{SL}(2, \C)$--connection $\nabla$ such that the eigenvalues of the residues
lie in $]-\tfrac{1}{2},\, \tfrac{1}{2}[$. It induces a parabolic structure on the underlying holomorphic
vector bundle $V$. The parabolic divisor is $\mathbf D\,=\, p_1 +\ldots + p_n$, where $p_j$ are the 
singular points of the connection. The parabolic weight $\rho_j$ at $p_j$ is the positive eigenvalue of
$\text{Res}_{p_j}(\nabla)$, and the quasiparabolic line at $p_j$ is the eigenline of $\text{Res}_{p_j}(\nabla)$
for the eigenvalue $\rho_j$.

Two non-resonant $\rm{SL}(2, \C)$--connection on $V$ with same  weights $\rho_j \in ]0, \tfrac{1}{2}[$ induce the same parabolic structure $\mathcal P$ if and only if they differ by a a {\it strongly parabolic Higgs field} on $(V,\, {\mathcal P})$.  
Recall that
a strongly parabolic Higgs field on $(V,\,{\mathcal P})$ is a trace free holomorphic section
$$
\Theta\, \in\, H^0(\Sigma,\, \text{End}(V)\otimes K_{\Sigma}\otimes {\mathcal O}_{\Sigma}({\mathbf D}))
$$
such that $$\Theta(p_j)(V_{p_j})\, \subset\, L_j\otimes (K_{\Sigma}\otimes {\mathcal O}_{\Sigma}({\mathbf 
D}))_{p_j}$$ for all $1\, \leq\, j\, \leq\, n$. These conditions imply that $\Theta(p_j)$ is nilpotent  and  the quasiparabolic line $L_j$ lies in the kernel of $\Theta(p_j)$, for all 
$1\, \leq\, j\, \leq\, n$.

\section{Logarithmic connections on $\mathcal S_4$}

Consider the Riemann sphere $\C P^1$ with three unordered marked points $\{0,\,1,\,\infty \}$
$${\mathcal S}_3\,=\, (\CP^1,\, \{0,\,1,\,\infty \})$$
and let 
\[S_3\, :=\,\CP^1\setminus\{0,\,1,\,\infty\}\]
be the three-punctured sphere. Fix a base point $p\, \in\, S_3$ and elements
$$
\gamma_0,\, \gamma_1\, \in\, \pi_1(S_3,\, p)
$$
such that $\gamma_0$ (respectively, $\gamma_1$) is the free homotopy class of the oriented loop around
the puncture $0$ (respectively, $1$). Then
$$\gamma_\infty\, :=\, (\gamma_1\gamma_0)^{-1}$$
is the free homotopy class of the oriented loop around the puncture $\infty$.

\subsection{Hyperbolic triangle and uniformization of the orbifold sphere}\label{triangle group}
$\,$\\
Consider ${\mathcal S}_3$ equipped with an orbifold structure, i.e., we assign to each marked point an angle 
$\alpha_i\,=\,\frac{2 \pi} {k_i}$, $i \,\in\, \{0,\,1,\,\infty \}$, where $k_i\, >\, 1$ are integers. Assume 
that $\frac{1}{k_0} + \frac{1}{k_1} + \frac{1} {k_{\infty} }\,<\, 1$. A hyperbolic uniformization of ${\mathcal 
S}_3$ equipped with the above orbifold structure is given by the following construction which goes back to the 
work of Schwarz, Klein and Poincar\'e (see \cite[Chapter~VI]{StG}).

The group $ \mathrm{PSL}(2,{\mathbb R})\,\subset\,\mathrm{PSL}(2,{\mathbb C})$ acts by M\"obius 
transformations on the upper half plane ${\mathcal H}^2\,:=\, \{z\,\in\, {\mathbb C}\,\mid\, {\rm Im}\, z\, 
>\, 0\}$. By viewing the upper half plane as the hyperbolic plane, $ \mathrm{PSL}(2,{\mathbb R})$ is in fact the group of orientation preserving isometries of ${\mathcal H}^2 $. Up to orientation preserving isometries, there exists a unique hyperbolic triangle $T$ in ${\mathcal H}^2$ with prescribed 
angles $(\frac{ \pi }{k_0},\,\frac{ \pi }{k_1}, \,\frac{ \pi } {k_{\infty} })$ \cite[Proposition IX.2.6]{StG}. Denote 
by $p_0,\,p_1 ,\, p_{\infty}\, \in\, {\mathcal H}^2$ the corresponding (ordered) vertices of $T$.

Denote by $\sigma_0,\,\sigma_1,\,\sigma_{\infty}$ the hyperbolic reflections across the geodesic arcs $(p_1,\,
p_{\infty})$, $(p_0,\, p_{\infty})$ and $(p_0,\, p_1)$ respectively. They generate a discrete subgroup of isometries of 
$ {\mathcal H}^2$. Consider its index two subgroup $\Gamma$ generated by 
$m_0\,=\,\sigma_{\infty} \circ \sigma_1$, $m_1\,=\,\sigma_0\circ\sigma_{\infty}$ and $m_{\infty}\,=\,\sigma_1\circ\sigma_0$. Geometrically, $\Gamma$ is generated by an even number of reflections across every geodesic edge of a hyperbolic 
geodesic triangle $T$; it is called a hyperbolic triangle group.
It is classical that such $\Gamma\,\subset\,\mathrm{PSL}(2,{\mathbb R})$ is a Fuchsian subgroup with a 
fundamental quadrilateral in $\mathcal{H}^2$ given by $P\,=\,T \cup \sigma_1(T)$. The vertices of $P$ are the points 
$p_0,\, p_1,\, p_{\infty},\, p_2$ with $p_2\,: =\, \sigma_1(p_1)$ (see \cite[Theorem VI.1.10 and Section VI.2.1]{StG}).
 
The oriented geodesic edges of $P$ satisfy  
$$m_0 \cdot (p_0,\,p_1)\,=\, (p_0,\, p_2) \quad \text{ and }\quad 
m_{\infty} \cdot (p_{\infty},\, p_2)\,= \,(p_{\infty}, p_1).$$
The maps $m_0$, $m_1$ and $m_{\infty}$ are of order $k_0,\, k_1$ and $k_{\infty},$ respectively, and $m_\infty \circ m_1 \circ m_0\,= 
\,\text{Id}$ by construction. 
Therefore, the hyperbolic triangle group $\Gamma$ generated by $m_0$, $m_1$ and $m_{\infty}$ satisfies
$$m_\infty \circ  m_1 \circ m_0\,=\, \text{Id}\ \quad \text{ and }\ \quad m_0^{k_0} = m_1^{k_1}= m_\infty^{k_\infty} \,=\,\text{Id}.$$

The quotient of ${\mathcal H}^2$ by the above Fuchsian hyperbolic triangle group $\Gamma$ endows ${\mathcal 
S}_3$, equipped with the orbifold structure $(\frac{ 2\pi }{k_0},\, \frac{ 2\pi }{k_1},\, \frac{ 2 \pi } {k_{\infty} 
})$ at the points $\{0,\,1,\,\infty \}$ respectively, with a compatible hyperbolic structure \cite[Chapter VI and Section 
VI.2.1]{StG}. In particular, the monodromy around the punctures $p_0,\, p_1,\, p_{\infty}$ of this uniformizing 
hyperbolic structure coincides with the rotations by the angles $\frac{ 2\pi }{k_0},\, \frac{ 2\pi }{k_1},\, \frac{ 2 
\pi } {k_{\infty} },$ respectively.

\subsection{Logarithmic connection on trivial bundle with Fuchsian monodromy}$\,$\\
For $\wt{\rho}\,\in\, ]0,\, \tfrac{1}{2}[$ fixed,
consider the logarithmic connection 
on the trivial holomorphic vector bundle ${\mathcal O}^{\oplus 2}_{\CP^1}$
\begin{equation}\label{c1}
\nabla\,=\,d+\begin{pmatrix}\frac{1}{8}&0\\0&-\frac{1}{8}\end{pmatrix}\frac{dz}{z}+
\begin{pmatrix}-4\wt\rho^2&1\\
\wt\rho^2-16\wt\rho^4&4\wt\rho^2\end{pmatrix}\frac{dz}{z-1}. 
\end{equation}
Since the singular locus of $\nabla$ is $\{0,\,1,\,\infty \}$, we consider $\nabla$ as a
logarithmic connection on ${\mathcal S}_3$. Throughout the paper we will use the convention that the marked points of a Riemann surface are the singular points of a logarithmic connection and branch  points of coverings. Further, let
$$
M_{\wt{\rho}}\, :\, \pi_1(S_3,\, p)\, \longrightarrow\, \mathrm{SL}(2,{\mathbb C})
$$
be the monodromy representation of the flat connection $\nabla$ in \eqref{c1}.

\begin{lemma}\label{lem:realSU3}
With the above notation the monodromy representation $M_{\wt{\rho}}$
of $\nabla$ in \eqref{c1} is conjugate to an irreducible $\mathrm{SU}(2)$ representation for
$\wt\rho\,<\,\tfrac{1}{4}$, and  to
an irreducible $\mathrm{SL}(2,\R)$ representation for $\tfrac{1}{4}\,<\, \wt\rho\,<\,\tfrac{1}{2}$.

If $\wt\rho\,=\, \frac{k-1}{2k}\,\in\,(\tfrac{1}{4},\,\tfrac{1}{2}),$ with $k\,\in\, \N^{>2}$,  the 
monodromy representation $M_{\wt{\rho}}$
is conjugated to the monodromy of a hyperbolic structure uniformizing ${\mathcal S}_3$ 
equipped with the orbifold structure $(\frac{\pi}{2},\,\frac{2\pi}{k},\,\frac{\pi}{2})$ at
points $\{0,\,1,\,\infty \}$ respectively. In particular, the image 
of the monodromy representation is the Fuchsian group generated by an even number of reflections across the 
geodesic edges of the hyperbolic geodesic triangle with angles $(\frac{\pi}{4},\,\frac{\pi}{k},\,\frac{\pi}{4}).$
\end{lemma}

\begin{proof}
Let $X_0,\,X_1$ and $X_\infty$  denote
the elements $M_{\wt{\rho}}(\gamma_0),\, M_{\wt{\rho}}(\gamma_1)$ and $M_{\wt{\rho}}(\gamma_\infty)$ of
$\mathrm{SL}(2,{\mathbb C})$ respectively. Moreover, let $R_0,\, R_1,\, R_\infty$ denote the respective residues of $\nabla$ at $0,\,1,\,\infty$.
Note that for $\wt \rho \,\in\, ]0,\, \tfrac{1}{2}[$ none of the eigenvalues of $R_0,\, R_1,\, R_\infty$ lies in $\tfrac{1}{2}\Z$; in other words,
$\nabla$ is non-resonant. Consequently, 
the conjugacy class of $X_i$ is given by
\begin{equation}\label{r1}
\exp(-2\pi\sqrt{-1}R_i)
\end{equation}
for $i\,=\, 0,\,1,\,\infty$ (see \cite[p.~53, Th\'eor\`eme 1.17]{De}). For $\nabla$ in \eqref{c1} we therefore compute 
\begin{equation}\label{e1}
\Tr(X_0)\,=\,\sqrt{2}\,=\,\Tr(X_\infty),\ \ \ \Tr(X_1)\,=\,2\cos(2\pi\wt\rho).
\end{equation}
This gives that the representation $M_{\wt{\rho}}$ is irreducible for
\[0\,<\,\wt\rho\,<\,\tfrac{1}{4}\quad\text{and}\quad \tfrac{1}{4}\,<\,\wt\rho\,<\,\tfrac{1}{2},\]
see \cite[p.~574, Proposition 4.1 (iii)]{Gold88}. It also follows that the three equations in \eqref{e1} determine
$M_{\wt{\rho}}$ uniquely up to the conjugation by an element of $\mathrm{SL}(2,{\mathbb C})$
\cite[p.~574, Proposition 4.1 (iv\,\&\,v)]{Gold88}.

Consider the matrices
\[
\wt{X}_0\,=\,\begin{pmatrix}\frac{1}{\sqrt{2}}& -\frac{1}{\sqrt{2}}\\ \frac{1}{\sqrt{2}}
& \frac{1}{\sqrt{2}}
\end{pmatrix}\]
\[\wt{X}_1=\begin{pmatrix}
\cos{2\pi\widetilde\rho}-\sqrt{(-1+\cos{2\pi\widetilde\rho})\cos{2\pi\widetilde\rho}}&
1-\cos{2\pi\widetilde\rho}+\sqrt{(-1+\cos{2\pi\widetilde \rho})\cos{2\pi\widetilde\rho}}\\
-1+\cos{2\pi\widetilde\rho}+\sqrt{(-1+\cos{2\pi\widetilde\rho})\cos{2\pi\widetilde\rho}}&
\cos{2\pi\widetilde\rho}+\sqrt{(-1+\cos{2\pi\widetilde\rho})\cos{2\pi\widetilde\rho}}
\end{pmatrix}\]
\[\wt{X}_\infty=\begin{pmatrix}\frac{1}{\sqrt{2}}
&\frac{-1+2\cos{2\pi\widetilde \rho}-2\sqrt{(-1+\cos{2\pi\widetilde \rho})\cos{2\pi\widetilde\rho}}}{\sqrt{2}}\\
\frac{1-2\cos{2\pi\widetilde\rho}-2\sqrt{(-1+\cos{2\pi\widetilde \rho})\cos{2\pi\widetilde \rho}}}{\sqrt{2}}&\frac{1}{\sqrt{2}} 
\end{pmatrix}.\]
These determine a monodromy homomorphism
$$
M'\, :\, \pi_1(S_3,\, p)\, \longrightarrow\, \mathrm{SL}(2,{\mathbb C})
$$
that takes $\gamma_0,\,\gamma_1$ and $\gamma_\infty$ to
$\wt{X}_0$, $\wt{X}_1$ and $\wt{X}_\infty$ respectively. Since the three equations in \eqref{e1} determine $M_{\wt{\rho}}$ uniquely up to conjugation
by some element of $\mathrm{SL}(2,{\mathbb C})$, we conclude that $M_{\wt{\rho}}$ and $M'$ are conjugate to each other. Evidently, the image of $M'$ lies in $\mathrm{SU}(2)$ if $0\,<\,\wt\rho\,<\,\tfrac{1}{4}$,
and it lies in $\mathrm{SL}(2,\R)$ if $\tfrac{1}{4}\,<\, \wt\rho\, <\,\tfrac{1}{2}$, proving the first part of the lemma.

To prove the second part, fix $k\,\in\,\N^{\geq3}$ and consider the special case of $\wt\rho\,=\, 
\frac{k-1}{2k}\,\in\,(\tfrac{1}{4},\,\tfrac{1}{2}).$ In this case the corresponding ${\rm SL}(2,\R)$ matrices generating the monodromy group $\Lambda$ for $\nabla$ specialize to
\[
\widetilde{X}_0\,=\,\begin{pmatrix}\frac{1}{\sqrt{2}}& -\frac{1}{\sqrt{2}}\\ \frac{1}{\sqrt{2}}
& \frac{1}{\sqrt{2}}
\end{pmatrix}\]
\[\widetilde{X}_1\,=\,\begin{pmatrix}
-\cos{\frac{\pi}{k}}-\sqrt{(1+\cos{\frac{\pi}{k}}) \cos\frac{\pi}{k}}&
1 + \cos{\frac{\pi}{k}}+\sqrt{(1+\cos{\frac{\pi}{k}})\cos{\frac{\pi}{k}}}\\
-1 -\cos{\frac{\pi}{k}}+\sqrt{(1+\cos{\frac{\pi}{k}})\cos{\frac{\pi}{k}}}&
-\cos{\frac{\pi}{k}}+\sqrt{(1+\cos{\frac{\pi}{k}}) \cos\frac{\pi}{k}}
\end{pmatrix}\]
\[\widetilde{X}_\infty\,=\,\begin{pmatrix}
\frac{1}{\sqrt{2}}&\frac{-1- 2\cos{\frac{\pi}{k}}-2 \sqrt{(1+\cos{\frac{\pi}{k}})\cos{\frac{\pi}{k}}}}{\sqrt{2}}\\
\frac{1+ 2\cos{\frac{\pi}{k}}-2\sqrt{(1+\cos{\frac{\pi}{k}})\cos{\frac{\pi}{k}}}}{\sqrt{2}}& \frac{1}{\sqrt{2}}
\end{pmatrix}.\]

Let $m(\widetilde{X}_0)$, $m(\widetilde{X}_1)$ and $m(\widetilde{X}_\infty)$ be the
automorphisms of the upper half plane
${\mathcal H}^2$
given by $\widetilde{X}_0$, $\widetilde{X}_1$ and $\widetilde{X}_\infty$ respectively.
The points of ${\mathcal H}^2$
\[
\begin{split}
p_0&\,=\,\sqrt{-1}\\
p_1&\,=\,\frac{1+\cos\tfrac{\pi}{k}+\sqrt{\cos\tfrac{\pi}{k}(1+\cos\tfrac{\pi}{k})}}{\sqrt{\cos\tfrac{\pi}{k}(1+\cos\tfrac{\pi}{k})}- \sqrt{-1}\sin\tfrac{\pi}{k}}\\
p_\infty&\,=\,\sqrt{-1}(\,1+2\cos\tfrac{\pi}{k}+2\sqrt{\cos\tfrac{\pi}{k}(1+\cos\tfrac{\pi}{k})}\,)
\end{split}
\]
are fixed by $m(\widetilde{X}_0)$, $m(\widetilde{X}_1)$ and $m(\widetilde{X}_\infty)$ respectively.
Recall that an element of $\text{PSL}(2, \R)$ is completely determined by a fixed point in ${\mathcal H}^2$
together with the differential at the fixed point. The differentials of
$m(\widetilde{X}_0)$, $m(\widetilde{X}_1)$ and $m(\widetilde{X}_\infty)$ at $p_0$, $p_1$ and
$p_\infty,$ respectively, are rotations and a short computation shows that these are given by
\begin{equation}\label{rot}
\begin{split}
D_{p_0} m(\widetilde{X}_0)&\,=\,-\sqrt{-1}\\
D_{p_1} m(\widetilde{X}_1) &\,=\,e^{-\tfrac{2\pi \sqrt{-1}}{k}}\\
D_{p_\infty} m(\widetilde{X}_\infty) &\,=\,-\sqrt{-1}.\\
\end{split}
\end{equation}
Therefore $\Lambda$ is conjugated in $\text{PSL}(2, \R)$ to the Fuchsian hyperbolic triangle group
associated to the hyperbolic triangle $(p_0,\, p_1,\, p_{\infty})$. The transformations 
$m(\widetilde{X}_0), m(\widetilde{X}_1)$ and $m(\widetilde{X}_\infty )$ coincide with $m_0$, $m_1$ and 
$m_{\infty}$ defined in Section \ref{triangle group},respectively (see also \cite[Chapter VI]{StG}).

{}From \eqref{rot}, the internal angles of the hyperbolic triangle are
\[
\alpha_0\,=\,\tfrac{\pi}{4}\quad 
\alpha_1\,=\,\tfrac{\pi}{k}\quad \text{ and } \quad 
\alpha_\infty\,=\,\tfrac{\pi}{4}.
\]
It follows from Section \ref{triangle group} that the monodromy homomorphism of $\nabla$ is conjugated in 
$\text{PSL}(2, \R)$ to the monodromy homomorphism of the uniformizing hyperbolic structure of the orbifold 
${\mathcal S}_3$ with angles $(\frac{\pi}{2},\,\frac{2\pi}{k},\,\frac{\pi}{2})$ at the points $\{0,\,1,\,\infty \}$
respectively.
\end{proof}

Let ${\mathcal S}_4$ denote the Riemann sphere $\CP^1$ with unordered four marked points
$$\{1,\,\sqrt{-1},\,-1,\,-\sqrt{-1}\}$$ and let
\begin{equation}\label{s4}
S_4\, :=\, \CP^1\setminus\{1,\,\sqrt{-1},\,-1,\,-\sqrt{-1}\}
\end{equation}
be the four-punctured sphere. Similarly, denote by ${\mathcal S}_6$ the Riemann sphere $\CP^1$ with six unordered
marked points
$\{0, 1,\,\sqrt{-1},\,-1,\,-\sqrt{-1}, \infty\}$, and define $S_6\, :=\, S_4 \setminus \{0, \infty \}.$ Consider the map
\[f\,\colon\, {\mathcal S}_6 \, \longrightarrow\, {\mathcal S}_3;\ \ z\,\longmapsto\, z^4 .\]
For the logarithmic connection $\nabla$ in \eqref{c1}, let
\begin{equation}\label{eq:wtnabla}
\nabla^1\, :=\, f^*\nabla
\end{equation}
be the logarithmic connection on the trivial holomorphic bundle
${\mathcal O}^{\oplus 2}_{{\mathcal S}_6}$ whose singular points coincide with the marked points.
We will construct a logarithmic connection on ${\mathcal O}^{\oplus 2}_{{\mathcal S}_4}$ using
$\nabla^1$.

Let $X$ denote $\CP^1$ with the ten unordered marked point $a_1,\, \cdots,\, a_{10}$
such that $$a^2_i \, \in\, \{0, 1,\,\sqrt{-1},\,-1,\,-\sqrt{-1}, \infty\}.$$ Let
\begin{equation}\label{vp}
\varpi\, :\, X\,\longrightarrow\, {\mathcal S}_6\, ,\ \ \ w\, \longmapsto\, w^2
\end{equation}
be the ramified covering map.
We have the logarithmic connection $\varpi^*\nabla^1$
on $\varpi^*{\mathcal O}^{\oplus 2}_{{\mathcal S}_6}\,=\, {\mathcal O}^{\oplus 2}_X$, where $\nabla^1$
is defined in \eqref{eq:wtnabla}. 
The Galois group of the ramified covering map $\text{Gal}(\varpi) \,=\,{\mathbb Z}/2{\mathbb Z}$
in \eqref{vp} acts on the vector bundle ${\mathcal O}^{\oplus 2}_X$; this action of ${\mathbb Z}/
2{\mathbb Z}$ on ${\mathcal O}^{\oplus 2}_X$ evidently preserves the logarithmic connection $\varpi^*\nabla^1.$

Let $z$ denote the
standard holomorphic coordinate on $S_6\, \subset\, \CP^1\setminus\{0,\, \infty\}$, so $\sqrt{z}\,:=\,
z\circ\varpi$ is a nowhere
vanishing holomorphic function on $\varpi^{-1}(S_6)\, \subset\, X$. For notational convenience,
we denote the subset $\varpi^{-1}(S_6)\, \subset\, X$ by $X'$.
Consider the holomorphic automorphism (= gauge transformation)
\[G\,:=\,\begin{pmatrix}\frac{1}{\sqrt{z}}&\frac{4}{-1+16\wt\rho^2}\sqrt{z}\\0&\sqrt{z}\end{pmatrix}
\begin{pmatrix}\frac{1}{\sqrt{\wt\rho(1-16\wt\rho^2)}}&0\\0&\sqrt{\wt\rho(1-16\wt\rho^2)}\end{pmatrix}\]
of ${\mathcal O}^{\oplus 2}_{X'}$ and let
\begin{equation}\label{n2}
\nabla^2\, :=\, ((\varpi^*\nabla^1)\vert_{X'}) . G \,=\, ((\varpi^*f^*\nabla)\vert_{X'}). G
\end{equation}
be the holomorphic connection on ${\mathcal O}^{\oplus 2}_{X'}$ given by the action of the automorphism $G$ on
the connection $\varpi^*\nabla^1\vert_{X'}$ (the connection $\nabla^1$ is defined in \eqref{eq:wtnabla}).

Although the above mentioned action $\text{Gal}(\varpi)\,=\, {\mathbb Z}/2{\mathbb Z}$
on $\varpi^*{\mathcal O}^{\oplus 2}_{S_6}\,=\, {\mathcal O}^{\oplus 2}_{X'}$ does not preserve $G$,
it is straightforward to check that the action of ${\mathbb Z}/2{\mathbb Z}$
on ${\mathcal O}^{\oplus 2}_{X'}$ actually preserves the connection $\nabla^2$ defined in \eqref{n2}.
Indeed, the action of the nontrivial element of ${\mathbb Z}/2{\mathbb Z}$ takes $G$ to $-G$. On the
other hand, the action of $-I\, \in\, \text{SL}(2,{\mathbb C})$ fixes every connection
on the trivial bundle ${\mathcal O}^{\oplus 2}_{X'}$. These imply that
the action of ${\mathbb Z}/2{\mathbb Z}$ preserves the connection $\nabla^2$. Hence there is a unique
holomorphic connection on ${\mathcal O}^{\oplus 2}_{S_6}$ whose pullback, by $\varpi$, is the
connection $\nabla^2$ on ${\mathcal O}^{\oplus 2}_{X'}\,=\, \varpi^*{\mathcal O}^{\oplus 2}_{S_6}$. Let
$\wt\nabla$ be the unique
holomorphic connection on ${\mathcal O}^{\oplus 2}_{S_6}$ such that
$$
\varpi^*\wt\nabla\, =\, \nabla^2\,=\, ((\varpi^*f^*\nabla)\vert_{X'}). G\,.
$$
A computation shows that
\begin{equation}\label{wtnabla}
\wt\nabla
\,=\,d+\begin{pmatrix}0&4\wt\rho \\ 4\wt\rho z^2&0 \end{pmatrix}\frac{dz}{z^4-1}
\end{equation}
on ${\mathcal O}^{\oplus 2}_{{\mathcal S}_6}\,=\, {\mathcal O}^{\oplus 2}_{\CP^1}$. In particular,
$\wt\nabla$ is a logarithmic connection on ${\mathcal O}^{\oplus 2}_{{\mathcal S}_4}$,
because $0,\,\infty$ are regular points of $\wt\nabla$.

\begin{lemma}\label{lem:realSU4}
The monodromy representation of the flat connection $\wt\nabla$ on ${\mathcal O}^{\oplus 2}_{{\mathcal S}_4}$
in \eqref{wtnabla}
is conjugate to a $\mathrm{SU}(2)$ representation if
$0\,<\,\wt\rho\,<\,\tfrac{1}{4}$, and it is conjugate to a $\mathrm{SL}(2,\R)$ representation if
$\tfrac{1}{4}\,<\, \wt\rho\, <\,\tfrac{1}{2}$.

Moreover, if $\wt\rho\,=\, \frac{k-1}{2k}$ with $k\,\in\, \N^{>2}$, then
the monodromy representation for $\wt\nabla$
is conjugated to the monodromy of a hyperbolic structure uniformizing ${\mathcal S}_4$ 
equipped with the orbifold structure $(\frac{2\pi}{k},\,\frac{2\pi}{k},\,\frac{2\pi}{k}, \frac{2\pi}{k})$
at the four marked points. In particular, the image 
of the monodromy representation is an index 4 subgroup in the Fuchsian group generated by an even number of reflections across the 
geodesic edges of the hyperbolic geodesic triangle with angles $(\frac{\pi}{4},\,\frac{\pi}{k},\,\frac{\pi}{4}).$
\end{lemma}

\begin{proof}
The monodromy representation of a pulled-back flat connection is the pull-back of the monodromy representation.
Further, a gauge transformation of a flat connection does not change the conjugacy class of the
monodromy representation. Let ${\mathcal I}\, \subset\, \mathrm{SL}(2,\C)$ be the image of the monodromy
homomorphism for the connection $\nabla^1$ in \eqref{eq:wtnabla} (the conjugacy class of this subgroup is unique).
Let \[\widetilde{\mathcal I}\, \subset\, \mathrm{SL}(2,\C)\] be the subgroup generated by $I$ and $-I$.
Since the action of the nontrivial element of the structure group
${\mathbb Z}/2{\mathbb Z}$ of the principal bundle in \eqref{vp} takes $G$ to $-G$,
the image of the monodromy homomorphism of $\wt\nabla$ is contained in
$\widetilde{\mathcal I}$ by construction. Then the first assertion of the  lemma follows from Lemma \ref{lem:realSU3}.

For the second statement, let $\wt\rho\,=\, \frac{k-1}{2k}\,\in\,(\tfrac{1}{4},\,\tfrac{1}{2}),$ with $k\,\in\, \N^{>2}$. 
It was shown in Lemma \ref{lem:realSU3} that the monodromy homomorphism for $\nabla$ is 
conjugated to the monodromy homomorphism of the uniformizing hyperbolic structure of the orbifold ${\mathcal S}_3$, 
with angles $(\frac{\pi}{2},\,\frac{2\pi}{k},\,\frac{\pi}{2})$ at points $\{0,\,1,\,\infty \}$ respectively.

The monodromy homomorphism for $\wt\nabla$ is the pull-back of the monodromy homomorphism for $\nabla$ through a 
4-fold covering totally branched over the marked points $0,\, \infty\,\in\, {\mathcal S}_3$. 
Therefore, the monodromy homomorphism for $\wt\nabla$ is conjugated to the monodromy homomorphism of the uniformizing 
hyperbolic structure of the orbifold ${\mathcal S}_4$ with angles 
$(\frac{2\pi}{k},\,\frac{2\pi}{k},\,\frac{2\pi}{k},\, \frac{2\pi}{k})$ (i.e., the orbifold structures at the four
preimages of $1$ are same). The image of the monodromy homomorphism of $\wt\nabla$ is an index 4 subgroup in
the Fuchsian triangle group $\Lambda$ defined in the proof of Lemma \ref{lem:realSU3}.
\end{proof}

\section{Pullback to hyperelliptic Riemann surfaces}

Let $\Sigma_{k}$ be the compact Riemann surface of genus $k-1$ defined by the algebraic equation
\begin{equation}\label{Sigma_k}
Y^{k}\,=\, \frac{Z^2-1}{Z^2+1}\,.
\end{equation}
It has the projection of degree $k$
\begin{equation}\label{fk}
f_k\,\,\colon\,\,\Sigma_{k}\,\longrightarrow\, \CP^1\, ,\ \ \ (Y,\,Z)\,\longmapsto\, Z.
\end{equation} The  hyperelliptic involution is given by $(Y,\,Z)\,\longmapsto\, (Y,\,-Z).$ 
Note that for $k=2$, the elliptic curve $\Sigma_2$ is of square conformal type and we identify $\Sigma_2\,=\,\C/(2\Z+2\sqrt{-1}\Z)$.

For $k\,\in\,\N^{\geq3}$ let $\wt\rho\,=\frac{k-1}{2k}\,\in\,(\tfrac{1}{4},\,\tfrac{1}{2})$ and consider the logarithmic connection
\begin{equation}\label{eq:D}
D\,=\,d+\begin{pmatrix} \widetilde\rho&0\\0&-\widetilde\rho\end{pmatrix}
\left(\frac{dz}{z-1}-\frac{dz}{z+\sqrt{-1}}+\frac{dz}{z+1}-\frac{dz}{z-\sqrt{-1}}\right)
\end{equation}
on ${\mathcal O}^{\oplus 2}_{{\mathcal S}_4}$ over ${\mathcal S}_4$. Then $D$ can be pulled back to the 
logarithmic connection
\begin{equation}\label{fkd}
(f^*_k{\mathcal O}^{\oplus 2}_{{\mathcal S}_4},\, f^*_k D) \,=\, ({\mathcal O}^{\oplus 2}_{\Sigma_{k}},\,
f^*_k D)
\end{equation}
by the map $f_k$ in \eqref{fk}. The singular points of $f^*_k D$ are
\begin{equation}\label{p4}
p_1\,=\, (0,\, 1),\ p_2\,=\, (\infty,\, \sqrt{-1}),\ p_3\,=\, (0,\, -1),\ p_4\,=\, (\infty,\, - \sqrt{-1})
\end{equation}
in terms of the above pair of coordinate functions  $(Y,\, Z)$ on $\Sigma_{k}$. Let
\begin{equation}\label{sc}
\Sigma'_{k}\, :=\, \Sigma_{k}\setminus \{p_1,\, p_2,\, p_3,\, p_4\}
\end{equation}
be the complement of the points in \eqref{p4}. Then the following proposition holds.

\begin{proposition}\label{red}
Let $k\,\in\,\N^{\geq3}$ and $\wt\rho\,=\,\frac{k-1}{2k}\,\in\,(\tfrac{1}{4},\,\tfrac{1}{2})$.
\begin{enumerate}
\item If $k$ is odd, then there is a meromorphic automorphism $G$ of ${\mathcal O}^{\oplus 2}_{\Sigma_{k}}$
such that
\begin{itemize}
\item $G$ is nonsingular on $\Sigma'_{k}$,
\item $G$ gauges the holomorphic connection $(f_k^*D)\vert_{\Sigma'_{k}}$ 
 to the trivial
holomorphic connection on
${\mathcal O}^{\oplus 2}_{\Sigma'_{k}}.$ In particular, $f_k^*D$ has trivial monodromy.
\end{itemize}

\item If $k$ is even, then there is a holomorphic line bundle ${\mathcal L}$ over $\Sigma_{k}$ with a logarithmic
connection $\nabla^{\mathcal L}$ with polar part contained in $\mathbf D=p_1+p_2+p_3+p_4$ such that
\begin{itemize}
\item the image of the monodromy homomorphism of $\nabla^{\mathcal L}$ is $\{\pm 1\}\, \subset\, {\mathbb C}^*$, 
\item there is a meromorphic isomorphism
$$
G\, :\, {\mathcal O}^{\oplus 2}_{\Sigma_{k}} \, \longrightarrow\, {\mathcal O}^{\oplus 2}_{\Sigma_{k}}\otimes {\mathcal L}\, =\,
{\mathcal L}^{\oplus 2}, 
$$
singular at $\mathbf D$, which gauges 
the
holomorphic connection
$(f_k^*D)\otimes \nabla^{\mathcal L}$ on ${\mathcal L}^{\oplus2}\vert_{\Sigma'_{k}}$ to
the trivial holomorphic connection on ${\mathcal O}^{\oplus 2}_{\Sigma'_{k}}$.
In particular, the
monodromy of 
$(f_k^*D)\otimes \nabla^{\mathcal L}$  is trivial.
\end{itemize}
\end{enumerate}
\end{proposition}
\begin{remark}
Throughout the paper we use the convention that the tensor product 
of two connections $\nabla^1$ on $V^1$ and $\nabla^2$ on $V^2$ is the connection on $V^1\otimes V^2$ given by the operator
\[\nabla^1\otimes\nabla^2:=\nabla^1\otimes{\rm Id}+{\rm Id}\otimes\nabla^2.\]
\end{remark}
\begin{proof}
Equation \eqref{Sigma_k} gives that
\[d\log Y\,=\, \tfrac{1}{k}d\log\frac{Z^2-1}{Z^2+1}\, .\]
Thus, for $k$ odd,
\begin{equation}\label{ps}
\Psi\,=\,\begin{pmatrix} Y^{-\tfrac{k-1}{2}}&0\\0 &Y^{\tfrac{k-1}{2}}\end{pmatrix}
\end{equation}
is a well-defined global meromorphic frame of ${\mathcal O}^{\oplus 2}_{\Sigma_{k}}$ that satisfies the following:
\begin{enumerate}
\item the restriction $\Psi\vert_{\Sigma'_{k}}$ is a holomorphic frame of ${\mathcal O}^{\oplus 2}_{\Sigma'_{k}}$,

\item $\Psi\vert_{\Sigma'_{k}}$ is a parallel frame for the holomorphic connection
$(f_k^*D)\vert_{\Sigma'_{k}}$.
\end{enumerate}
To show that $\Psi\vert_{\Sigma'_{k}}$ is indeed parallel note that
$\frac{4z dz}{z^4-1}\,=\,\frac{dz}{z-1}-\frac{dz}{z+\sqrt{-1}}+\frac{dz}{z+1}-\frac{dz}{z-\sqrt{-1}}$. 
Let $G$ be the automorphism of
${\mathcal O}^{\oplus 2}_{\Sigma'_{k}}$ that takes the standard frame to the frame $\Psi\vert_{\Sigma'_{k}}$.
Then $G$ gauges $(f_k^*D)\vert_{\Sigma'_{k}}$ to
the trivial connection on ${\mathcal O}^{\oplus 2}_{\Sigma'_{k}}$, because $\Psi\vert_{\Sigma'_{k}}$ is a parallel frame for
$(f_k^*D)\vert_{\Sigma'_{k}}$. This proves the proposition for odd $k$. 

If $k$ is even, $\Psi$ in \eqref{ps} is no longer single valued. Nevertheless, we can still recover the
trivial connection on ${\mathcal O}^{\oplus 2}_{\Sigma'_{k}}$ by twisting the pull-back of $D$ to $\Sigma_k$
by an appropriate line bundle connection. The construction goes as follows. 
The values of $Y^{\tfrac{k-1}{2}}$ produce a nontrivial double covering
$$
\delta\, :\, \widetilde{\Sigma}\, \longrightarrow\, \Sigma_k
$$
branched over the subset $\{p_1,\, p_2,\, p_3,\, p_4\}$ in \eqref{p4}. Let
$$
\widetilde{\Sigma}'\,:=\, \delta^{-1}(\Sigma'_{k})\, \subset\, \widetilde{\Sigma}.
$$
So $\delta\vert_{\widetilde{\Sigma}'}\, :\, \widetilde{\Sigma}'
\, \longrightarrow\, \Sigma'_{k}$ is an unramified double covering. Now $\Psi$ produces a meromorphic frame $\widetilde{\Psi}$ of $\delta^*{\mathcal O}^{\oplus 2}_{\Sigma_{k}}\,=\,
{\mathcal O}^{\oplus 2}_{\widetilde{\Sigma}}$ such that
the restriction of $\widetilde{\Psi}$ to $\widetilde{\Sigma}'$ is a holomorphic frame of
${\mathcal O}^{\oplus 2}_{\widetilde{\Sigma}'}$. This frame $\widetilde{\Psi}\vert_{\widetilde{\Sigma}'}$
is parallel
for the flat connection $(\delta^*f^*_kD)\vert_{\widetilde{\Sigma}'}$
on ${\mathcal O}^{\oplus 2}_{\widetilde{\Sigma}'}$.

The Galois group $\text{Gal}(\delta)\,=\, {\mathbb Z}/2{\mathbb Z}$ for $\delta$ has a natural action on
$\delta^*{\mathcal O}^{\oplus 2}_{\Sigma_{k}}\,=\,
{\mathcal O}^{\oplus 2}_{\widetilde{\Sigma}}$. The action of the nontrivial element of
$\text{Gal}(\delta)\,=\, {\mathbb Z}/2{\mathbb Z}$ evidently takes the frame $\widetilde{\Psi}$
to $-\widetilde{\Psi}$. Therefore, the  holomorphic frame $\widetilde{\Psi}$ of ${\mathcal O}^{\oplus 2}_{\widetilde{\Sigma}'}$ does not
descend to a holomorphic frame of ${\mathcal O}^{\oplus 2}_{\Sigma'_{k}}.$
Note that the action of $\text{Gal}(\delta)$ on
${\mathcal O}^{\oplus 2}_{\widetilde{\Sigma}}$ preserves the logarithmic connection $\delta^*f^*_kD$.
We will now construct a suitable twist of $\wt \Psi$ that descends.

Consider the holomorphic line bundle
$$
\widetilde{\mathcal L}\, :=\, {\mathcal O}_{\widetilde{\Sigma}}
$$
equipped with the following action of $\text{Gal}(\delta)$: the nontrivial element $\alpha
\, \in\, \text{Gal}(\delta)$ acts as multiplication
by $-1$ over the involution $\alpha$, meaning $f\, \longmapsto\, -f\circ\alpha$,  for any locally
defined holomorphic function $f$ on $\widetilde{\Sigma}$.
(The notation $\widetilde{\mathcal L}$ is used for emphasizing the nontrivial action of $\text{Gal}(\delta)$.)
It has a holomorphic connection defined by the de Rham
differential; this connection, which  will be denoted by $\nabla^{\widetilde{\mathcal L}}$, is
preserved by the action of $\text{Gal}(\delta)$ on $\widetilde{\mathcal L}.$

Now consider the holomorphic vector bundle
\begin{equation}\label{cf}
{\mathcal F}\, :=\,
(\delta^*{\mathcal O}^{\oplus 2}_{\Sigma_{k}}) \otimes \widetilde{\mathcal L}\,=\,
{\mathcal O}^{\oplus 2}_{\widetilde{\Sigma}}\otimes \widetilde{\mathcal L}
\end{equation}
on $\widetilde{\Sigma}$. It has the meromorphic frame $\widetilde{\Psi}\otimes {\mathbf 1}$, where
$\mathbf 1$ denotes the constant function $1$. This frame $\widetilde{\Psi}\otimes {\mathbf 1}$
is holomorphic over $\widetilde{\Sigma}'$ and it is preserved by the action of $\text{Gal}(\delta)$ on ${\mathcal F}$ (recall that $\alpha\, \in
\text{Gal}(\delta)$ acts
as multiplication by $-1$ on both $\widetilde{\Psi}$ and ${\mathbf 1}$). With respect to the product connection
\begin{equation}\label{nf}
\nabla^{\mathcal F}\, :=\, (\delta^*f^*_kD)\otimes\nabla^{\widetilde{\mathcal L}}
\end{equation}
on ${\mathcal F}$, the holomorphic frame $(\widetilde{\Psi}\otimes {\mathbf 1})\vert_{\widetilde{\Sigma}'}$
is in fact parallel
on ${\mathcal F}\vert_{\widetilde{\Sigma}'}$.
The actions of $\text{Gal}(\delta)$ on $\widetilde{\mathcal L}$ and
$\delta^*{\mathcal O}^{\oplus 2}_{\Sigma_{k}}$ together produce an action of $\text{Gal}(\delta)$ on
the vector bundle $\mathcal F$ in \eqref{cf}. The logarithmic connection $\nabla^{\mathcal F}$
in \eqref{nf} is evidently invariant under this action of $\text{Gal}(\delta)$ on $\mathcal F$.

Define the invariant direct image
$$
{\mathcal L}\, :=\, (\delta_*\widetilde{\mathcal L})^{\text{Gal}(\delta)}\, \subset\,
\delta_*\widetilde{\mathcal L}
$$
for the action of $\text{Gal}(\delta)$ on $\delta_*\widetilde{\mathcal L}$.
It is a holomorphic line bundle on $\Sigma_k$ such that $\delta^*{\mathcal L}\,=\,
{\mathcal O}_{\widetilde{\Sigma}}(-q_1-q_2-q_3-q_4)$, where $q_i\,\in\, \widetilde{\Sigma}$ satisfies
$\delta(q_i)\,=\, p_i$. The connection $\nabla^{\widetilde{\mathcal L}}$
on $\widetilde{\mathcal L}$, being preserved by the action of $\text{Gal}(\delta)$ on
$\widetilde{\mathcal L}$, produces a logarithmic connection $\nabla^{\mathcal L}$ on ${\mathcal L}$; its
residue is $\tfrac{1}{2}$ at each marked point $p_i$.
Since the logarithmic connection $\nabla^{\widetilde{\mathcal L}}$ has trivial monodromy representation, and
the residues of $\nabla^{\mathcal L}$ are $\tfrac{1}{2}$, it follows that
the image of the monodromy homomorphism for the above logarithmic connection
$\nabla^{\mathcal L}$ on ${\mathcal L}$ is exactly $\{\pm 1\}\, \subset\, {\mathbb C}^*$.

The above construction of ${\mathcal L}$ from $\widetilde{\mathcal L}$ shows that
the pull-back bundle $\delta^*({\mathcal O}^{\oplus 2}_{\Sigma_{k}}\otimes {\mathcal L})\, =\,
\delta^*{\mathcal L}^{\oplus 2}$ is holomorphically isomorphic to ${\mathcal F}
\otimes{\mathcal O}_{\widetilde{\Sigma}}(-q_1-q_2-q_3-q_4)$ (see \eqref{cf})
by a $\text{Gal}(\delta)$--equivariant holomorphic isomorphism.

The logarithmic connection $\nabla^{\mathcal F}$ in \eqref{nf} descends to a logarithmic
connection on ${\mathcal O}^{\oplus 2}_{\Sigma_{k}}\otimes {\mathcal L}$, because $\nabla^{\mathcal F}$
is preserved by the action $\text{Gal}(\delta)$ on ${\mathcal F}$. This descended logarithmic
connection on ${\mathcal O}^{\oplus 2}_{\Sigma_{k}}\otimes {\mathcal L}$ clearly coincides with
$f^*_kD \otimes\nabla^{\mathcal L}.$

The meromorphic frame $\widetilde{\Psi}\otimes {\mathbf 1}$ of $\mathcal F$ descends to
a holomorphic frame of $({\mathcal O}^{\oplus 2}_{\Sigma_{k}}\otimes {\mathcal L})\vert_{\Sigma'_{k}}$
because $\widetilde{\Psi}\otimes {\mathbf 1}$ is preserved by the action of
$\text{Gal}(\delta)$. It was observed above that the holomorphic frame
$(\widetilde{\Psi}\otimes {\mathbf 1})\vert_{\widetilde{\Sigma}'}$ of
$({\mathcal O}^{\oplus 2}_{\Sigma_{k}}\otimes {\mathcal L})\vert_{\Sigma'_{k}}$ is parallel
with respect to the holomorphic connection $\nabla^{\mathcal F}\vert_{\widetilde{\Sigma}'}$ in \eqref{nf}.
Consequently, the holomorphic frame of
$({\mathcal O}^{\oplus 2}_{\Sigma_{k}}\otimes {\mathcal L})\vert_{\Sigma'_{k}}$ given by
$\widetilde{\Psi}\otimes {\mathbf 1}$
is parallel with respect to the holomorphic connection $f^*_kD\otimes\nabla^{\mathcal L}$
on $({\mathcal O}^{\oplus 2}_{\Sigma_{k}}\otimes {\mathcal L})\vert_{\Sigma'_{k}}$, completing the proof.
\end{proof}

The parabolic structure on ${\mathcal O}^{\oplus 2}_{{\mathcal S}_4}$ induced by the logarithmic connection $D$
in \eqref{eq:D} admits the strongly parabolic Higgs field
\begin{equation}\label{Phi}
\Phi\,=\,\begin{pmatrix}0& \frac{dz}{z-1}- \frac{dz}{z+1}\\ \frac{dz}{z-\sqrt{-1}}-
\frac{dz}{z+\sqrt{z-1}}&0\end{pmatrix}.
\end{equation}
The following lemma states that the singularities of $\Phi$ have the same behavior
under pull-back and gauge transformation as the connection $D$ itself.

\begin{lemma}\label{strong-pb}
Let $\Phi$ be the strongly parabolic Higgs field defined in \eqref{Phi}, $f_k$ the projection from $\Sigma_k$ to
$\mathcal S_4$ in \eqref{fk} and $G$ the gauge transformation in 
Proposition \ref{red}. Then
$$G^{-1}\circ f_k^*\Phi\circ G$$ extends to a holomorphic Higgs field on the trivial
holomorphic bundle ${\mathcal O}^{\oplus 2}_{\Sigma_{k}}$ over $\Sigma_k$.
\end{lemma}

\begin{proof}
As before we  have to distinguish between even $k$ and odd $k.$ For odd $k$ it is evident that
$G^{-1}\circ f_k^*\Phi\circ G$ is a holomorphic Higgs field on $\Sigma'_k$ with respect to the trivial holomorphic structure induced by
$d\,=\,(f_k^*D).G$. We have to show that $G^{-1}\circ f_k^*\Phi\circ G$ is holomorphic
at  the branch points of $f_k.$
Consider $p_1\,=\,f_k^{-1}(1).$ Then the pull-back of $\Phi$, considered as an endomorphism-valued 1-form, is
meromorphic and  of the form 
\[f_k^*\Phi\,=\,\begin{pmatrix}0&b\\c&0\end{pmatrix}.\] 
The diagonal entries of the pull-back vanish identically, while the lower left entry $c$ has a zero of order
$k-1$ at $p_1$ as $f_k$ is totally branched. The upper right entry $b$ has a pole of order 1  at $p_1$.
Since the meromorphic function $Y$ (of degree 2) on $\Sigma_k$ (see (\ref{Sigma_k})) has a zero of order 1 at
$p_1$, and $k\,>\,1,$
\[G^{-1}\circ f_k^*\Phi\circ G\,=\,\begin{pmatrix} 0& b Y^{k-1}\\ cY^{-k+1}&0\end{pmatrix}\]
is holomorphic  at $p_1$.
The same argument works for the other branch points $p_2,\,p_3,\,p_4$ of $f_k$ showing that
$G^{-1}\circ f_k^*\Phi\circ G$ is a holomorphic Higgs field on the trivial holomorphic bundle.

When $k$ is even, we consider 
\[f_k^*\Phi \,\cong \, f_k^*\Phi \otimes \mathbf{1}\]
as an  endomorphism-valued 1-form on the vector bundle
${\mathcal O}^{\oplus 2}_{\Sigma_{k}}\otimes {\mathcal L}$
over $\Sigma'_k.$ It is holomorphic with respect to the holomorphic structure induced by the connection
 $(f_k^*D)\otimes\nabla^{\mathcal L}.$
The same arguments as for $k$ odd then show that
$G^{-1}\circ f_k^*\Phi\circ G$ extends to a holomorphic endomorphism-valued 1-form on the
trivial bundle ${\mathcal O}^{\oplus 2}_{\Sigma_{k}}$ over $\Sigma_k.$
\end{proof}

The  following proposition and its proof are similar to some results about symmetric minimal
surfaces in the 3-sphere \cite[Section 3.3]{HHSch}.

\begin{proposition}\label{FRS}
Let $k\,\in\, \N^{>2}$ and $\wt\rho\,=\, \frac{k-1}{2k}.$ Consider the 
logarithmic connection $\wt{\nabla}$ on ${\mathcal O}^{\oplus 2}_{{\mathcal S}_4}$
given in \eqref{wtnabla} and its pull-back $f_k^*\wt\nabla$ on ${\mathcal O}^{\oplus 2}_{\Sigma_k}$
with polar part in $\mathbf D=p_1+p_2+p_3+p_4$.
Then the parabolic structure associated to $\wt \nabla$ is unstable. Furthermore,

\begin{enumerate}
\item if  $k$ is odd 
\begin{itemize}
\item there exists a flat $C^\infty$ connection on ${\mathcal O}^{\oplus 2}_{\Sigma_{k}}\, 
\longrightarrow\,\Sigma_k$ which is $C^\infty$ gauge equivalent to
$(f_k^*\wt\nabla)\vert_{\Sigma'_{k}}$ over $\Sigma'_{k}$. In particular, $f_k^*\wt\nabla$ has trivial local 
monodromy around the singular points $\{ p_1,\, p_2,\, p_3,\, p_4\}$;
\item the monodromy homomorphism of  $f_k^*\wt\nabla$ is the one of the uniformizing hyperbolic structure of $\Sigma_{k}$, in particular it is Fuchsian.  
\end{itemize}

\item If $k$ is even, then there is a holomorphic line bundle ${\mathcal L}$ over $\Sigma_{k}$ with a logarithmic
connection $\nabla^{\mathcal L}$ with polar part in $\mathbf D=p_1+p_2+p_3+p_4$ such that
\begin{itemize}

\item the image of the monodromy homomorphism for $\nabla^{\mathcal L}$ is
$\{\pm 1\}\, \subset\, {\mathbb C}^*$;

\item
 there exists  a
 $C^\infty$ vector bundle isomorphism 
 $$
G\, :\, 
{\mathcal O}^{\oplus 2}_{\Sigma'_{k}}\, \longrightarrow\,
{\mathcal O}^{\oplus 2}_{\Sigma'_{k}}\otimes {\mathcal L}\, =\,
{\mathcal L}^{\oplus 2} 
$$
 over $\Sigma'_{k}\, 
\subset\, \Sigma_{k}$ 
which gauges $((f_k^*\wt\nabla)\otimes\nabla^{\mathcal 
L})\vert_{\Sigma'_{k}}$
to
 a $C^\infty$ flat connection $\widehat{\nabla}$ with Fuchsian monodromy on the trivial bundle over
${\Sigma_{k}}$;

\item   $f_k^*\wt\nabla\otimes \nabla^{\mathcal 
L}$ on ${\mathcal L}^{\oplus 2}$ has trivial local monodromy around the 
singular points $\{ p_1,\,p_2,\,p_3,\, p_4\}$, and its monodromy representation
coincides with the monodromy homomorphism of the uniformizing hyperbolic structure of $\Sigma_{k}$.
\end{itemize}
\end{enumerate}
\end{proposition}

\begin{remark} The holomorphic line bundle $\mathcal L$ and the logarithmic connection $\nabla^{\mathcal L}$ in the statement of Proposition \ref{FRS} (2) are the
same as in the statement of Proposition \ref{red} (2).
\end{remark}

\begin{proof}[{Proof of Proposition \ref{FRS}}]
The logarithmic connection $\wt\nabla$ on ${\mathcal S}_4$ in \eqref{wtnabla} is 
\begin{equation}\label{wtnabla2}
\wt\nabla\,=\,d+\begin{pmatrix}0&\wt\rho z^{-1} \\\wt\rho z&0 \end{pmatrix} 
\left( \frac{dz}{z-1}-\frac{dz}{z+\sqrt{-1}}+\frac{dz}{z+1}-\frac{dz}{z-\sqrt{-1}} \right).
\end{equation}
At each point of the singular locus $\{1,\,-1,\, \sqrt{-1},\, -\sqrt{-1}\}$ the eigenvalues of the residue
of $\wt\nabla$ are $\wt\rho$ and $-\wt\rho$. Using \eqref{wtnabla2} we compute the eigenlines for the positive
eigenvalue $\wt\rho$ of the residues of $\wt\nabla$ at $x\in\{1,\,-1,\, \sqrt{-1},\, -\sqrt{-1}\}$ to be: 
\[l_x= {\mathbb C}\cdot (x,\, 1)\, \subset\, {\mathbb C}^2.\]
Recall from Section \ref{ss:parabolic}
that the quasiparabolic structures at $\{1,\,-1,\, \sqrt{-1},\, -\sqrt{-1}\}$ are given by the
eigenlines for the eigenvalue $\wt\rho$.
Let
\begin{equation}\label{bl}
{\mathcal O}^{\oplus 2}_{\CP^1}\, \supset\,{\mathbb L} \, \longrightarrow\, \CP^1
\end{equation}
be the tautological subbundle whose fiber over any $z\, \in\, \mathbb C$ is ${\mathbb C}\cdot (z,\, 1)$ and the
fiber over $\infty$ is ${\mathbb C}\cdot(1,\, 0)$.
Therefore, at each point $x$ of the singular locus $\{1,\,-1,\, \sqrt{-1},\, -\sqrt{-1}\}$
the subspace ${\mathbb L}_x\, \subset\, ({\mathcal O}^{\oplus 2}_{\CP^1})_x\,=\, {\mathbb C}^2$ coincides with
the eigenline of $\text{Res}_{x}(\wt\nabla)$ with respect to the eigenvalue $\wt\rho$.
Consequently, the parabolic degree of the line subbundle ${\mathbb L}\, \subset\, {\mathcal O}^{\oplus 2}_{\CP^1}$
in \eqref{bl}, with respect to the parabolic structure induced by $\wt\nabla$ is
\begin{equation}\label{spd}
\text{par-deg}({\mathbb L}) \,=\, \text{degree}({\mathbb L})+4\wt\rho\,=\, 4\wt\rho-1 \, >\, 0\,=\,
\text{par-deg}({\mathcal O}^{\oplus 2}_{\CP^1})\, .
\end{equation}
Therefore, ${\mathcal O}^{\oplus 2}_{\CP^1}$ equipped with the parabolic structure given by $\wt\nabla$ is unstable.

Consider the standard inner product on ${\mathbb C}^2$. It produces a constant Hermitian structure
on ${\mathcal O}^{\oplus 2}_{\CP^1}$ which is flat with respect to the
trivial holomorphic connection on ${\mathcal O}^{\oplus 2}_{\CP^1}$. Let $\mathbb L^\perp$ denote the orthogonal
complement of the line subbundle ${\mathbb L}$ in \eqref{bl}, so we have the $C^\infty$ decomposition
\begin{equation}\label{decomposeC2L}
{\mathcal O}^{\oplus 2}_{\CP^1}\,=\,{\mathbb L}\oplus {\mathbb L}^\perp.
\end{equation}
Note that ${\mathbb L}^\perp$ is identified with ${\mathcal O}^{\oplus 2}_{\CP^1}/{\mathbb L} =\,
{\mathbb L}^{^*}$, because $\bigwedge^2{\mathcal O}^{\oplus 2}_{\CP^1}\,=\,
{\mathcal O}_{\CP^1}$. With respect to the decomposition in \eqref{decomposeC2L},
the holomorphic structure of ${\mathcal O}^{\oplus 2}_{\CP^1}$,
which is the same as the $(0,\,1)$-part of the flat connection for $\wt \nabla$, is
\[
\overline{\partial}^{\wt\nabla} \,=\,\begin{pmatrix} \overline{\partial}^{\mathbb L}& \psi\\0& \overline{\partial}^{{\mathbb L}^*}\end{pmatrix}
\]
for some non-trivial $C^\infty$ section $\psi$ of $\overline{K}_{\CP^1}\otimes {\mathbb L}^{\otimes 2}$
over $\CP^1$, where
$\overline{\partial}^{\mathbb L}$ and $\overline{\partial}^{{\mathbb L}^{^*}}$ are the Dolbeault operators for
${\mathbb L}$ and ${\mathbb L}^{^*}$ respectively. The $(1,\, 0)$-part $\partial^{\wt\nabla}$ of $\wt\nabla$ is
\begin{equation}\label{wn10}
\partial^{\wt\nabla}\,=\,\begin{pmatrix} \partial^{\mathbb L}& \alpha\\ \varphi& \partial^{{\mathbb L}^*}\end{pmatrix},
\end{equation}
where
$\partial^{\mathbb L}$ is a $C^\infty$ $(1,\,0)$--connection on the holomorphic line bundle ${\mathbb L}$
over $S_4$ (defined in \eqref{s4}), and $\partial^{{\mathbb L}^{^*}}$ is the dual
$(1,\,0)$--connection on ${\mathbb L}^{^*}\vert_{S_4}$.
Furthermore, in \eqref{wn10} $\alpha$ is a $C^\infty$ section of $K_{\CP^1}\otimes {\mathbb L}^{\otimes 2}$ over $S_4$,
and $\varphi$ is a holomorphic
section of $K_{\CP^1}\otimes ({\mathbb L}^*)^{\otimes 2}$. In fact $\varphi$ is the second fundamental form of
the holomorphic subbundle
${\mathbb L}\, \subset\, {\mathcal O}^{\oplus 2}_{\CP^1}$ for the logarithmic connection $\wt\nabla$.
We note that $\varphi$ is holomorphic over the entire $\CP^1$ because at every singular point $q_l$ of
$\wt\nabla$, the fiber ${\mathbb L}_{q_l}\, \subset\, ({\mathcal O}^{\oplus 2}_{\CP^1})_{q_l}$ is an
eigenline of the residue of $\wt\nabla$.

If $\varphi\,=\, 0$, then the line subbundle $\mathbb L$ 
is preserved by $\wt\nabla$ which gives a contradiction since 
the parabolic degree $\mathbb L$
with respect to the induced parabolic structure is nonzero (see \eqref{spd} and \cite{Oh}). Hence we conclude
that $\varphi\,\not=\, 0$ and, by choosing a suitable holomorphic isomorphism
between $K_{\CP^1}$ and ${\mathbb L}^{\otimes 2}$,  we can normalize $\varphi$ to be the constant function $1$.

Consider the pulled back logarithmic connection $f_k^*\wt\nabla$ on the trivial holomorphic vector bundle 
${\mathcal O}^{\oplus 2}_{\Sigma_{k}}$ over $\Sigma_k$, where $f_k$ is the map in \eqref{fk}.
It is singular over the four branched points $p_1,\,\cdots,\,p_4$ in \eqref{p4}.
\subsection*{Case 1: $k$ is odd}$\;$\\
We desingularize $f_k^*\wt\nabla$ at $p_l$, $1\,\leq\,l\, \leq\, 4$, as follows. Take a holomorphic
coordinate function $z$ defined on an open neighborhood of $f_k(p_l)\,\in\,\CP^1$ with $z(f_k(p_l))\,=\,0$.
Let $y$ be a holomorphic coordinate function defined on an open subset $U_l\, \subset\, \Sigma_k$ 
containing $p_l$ such that $y^k\,=\,z\circ f_k$.
 Consider the meromorphic endomorphism
\begin{equation}\label{hl}
h_l\,=\, \begin{pmatrix} y^{-\frac{k-1}{2}}&0\\0& y^{\frac{k-1}{2}}\end{pmatrix}
\end{equation}
 of ${\mathcal O}^{\oplus 2}_{U_l}\,=\, {\mathcal O}^{\oplus 2}_{\Sigma_k}\Big\vert_{U_l}$.
It is a holomorphic automorphism over
$U'_l:=U_l\setminus\{p_l\}$.
Let $(f_k^*\wt\nabla)\vert_{U'_l}. (h_l\vert_{U'_l})$ be the holomorphic connection on ${\mathcal O}^{\oplus 
2}_{U'_l}$ produced by the action of the gauge transformation $h_l\vert_{U'_l}$ on the connection 
$(f_k^*\wt\nabla)\vert_{U'_l}$.

We claim that $(f_k^*\wt\nabla)\vert_{U'_l}. (h_l\vert_{U'_l})$ extends to a
$C^\infty$ connection on ${\mathcal O}^{\oplus 2}_{U_l}$. To prove the above claim, first note that the upper right entry of the connection
$(f_k^*\wt\nabla)\vert_{U'_l}. (h_l\vert_{U'_l})$ (with respect to the splitting 
$\mathbb L \oplus \mathbb L^\perp$) is multiplied with the function $y^{k-1}$ and is therefore smooth at $p_l$ 
(it vanishes at $p_l$ with some higher order). Moreover, the pull-back $f^*\varphi$ of the non-vanishing 1-form $\varphi$ 
with values in ${\mathcal O}_{\CP^1}(2)$ has vanishing order $k-1$ at $p_l$. Hence, the lower left entry of 
$(f_k^*\wt\nabla)\vert_{U'_l}. (h_l\vert_{U'_l})$ with respect to \eqref{decomposeC2L}, which becomes 
\[y^{1-k}f^*_k\varphi\, ,\] extends smoothly and non-vanishingly to $p_l$. This proves the claim.

Since $(f_k^*\wt\nabla)\vert_{U'_l}. (h_l\vert_{U'_l})$ extends to a $C^\infty$ connection on ${\mathcal 
O}^{\oplus 2}_{U_l}$, the local monodromy of $f_k^*\wt\nabla$ at each $p_l$ is trivial.

Now fix a global $C^\infty$ automorphism
\begin{equation}\label{hl2}
h\, \in\, C^\infty(\Sigma'_{k}, \, \text{Aut}({\mathcal O}^{\oplus 2}_{\Sigma'_{k}}))
\end{equation}

such that $\det h\,=\, 1$ and, for each $1\, \leq\, l\, \leq\, 4$,
it coincides with $h_l$ (see \eqref{hl}) on a neighborhood 
of $p_l$; such a global gauge $h$ does exist. From the above observation that
$(f_k^*\wt\nabla)\vert_{U'_l}. (h_l\vert_{U'_l})$ extends to a
$C^\infty$ connection on ${\mathcal O}^{\oplus 2}_{U_l}$ it follows immediately that $(f_k^*\wt\nabla). h$ is
a $C^\infty$ flat connection on the trivial $C^\infty$ vector bundle
$$ \Sigma_k\times {\mathbb C}^2\, =:\, E^0\, .$$
The holomorphic structure on $E^0$
given by the flat connection $(f_k^*\wt\nabla). h$ is not the trivial holomorphic
structure on ${\mathcal O}^{\oplus 2}_{\Sigma_k}$, as $h$ in \eqref{hl2} is not holomorphic.
In fact, we claim that it is a uniformization bundle on $\Sigma_k.$

Let ${\mathcal E}^0$ denote the holomorphic vector bundle over $\Sigma_k$
given by the holomorphic structure of $(f_k^*\wt\nabla). h$. Since $\det h\,=\, 1$, it follows that
${\mathcal E}^0$ is a holomorphic $\text{SL}(2,{\mathbb C})$--bundle with $(f_k^*\wt\nabla). h$ being a
holomorphic $\text{SL}(2,{\mathbb C})$--connection on it.

Consider the pulled back line bundle
$$
f_k^*{\mathbb L}\, \subset\, f_k^*{\mathcal O}^{\oplus 2}_{\CP^1}\,=\, {\mathcal O}^{\oplus 2}_{\Sigma_k}\, ,
$$
where $\mathbb L$ is the tautological bundle constructed in \eqref{bl}.
Note that $h(f_k^*{\mathbb L})\, \subset\, {\mathcal E}^0\vert_{\Sigma'_k}$ is a holomorphic line subbundle
(recall that $h$ in \eqref{hl2} is defined only on $\Sigma'_k$). Since $h$ is meromorphic near each $p_l$
(as $h_l$ in \eqref{hl} is meromorphic around $p_l$ and $h$
coincides with $h_l$ around $p_l$), we conclude that $h(f_k^*{\mathbb L})=:\widetilde{\mathbb L}$ extends to a
holomorphic subbundle of ${\mathcal E}^0$ over the entire $\Sigma_k$.

For $1\,\leq\, l\, \leq\, 4$ fixed, 
let $s$ be a non-vanishing holomorphic section of ${\mathbb L}$
defined on an open subset $\wt U_l\, \subset\, \CP^1$ around  $f_k(p_l)$. Then the holomorphic section $f^*_k(s\vert_{\wt U_l\setminus\{f_k(p_l)\}})$
of $\widetilde{\mathbb L}\vert_{f^{-1}_k(\wt U_l\setminus\{f_k(p_l)\})}$ extends to a holomorphic section
of $\widetilde{\mathbb L}\vert_{f^{-1}_k(\wt U_l)}$ vanishing at $p_l$ with order $(k-1)/2$. Indeed, this
follows immediately from the expression of $h_l$ in \eqref{hl}. From this we conclude that
\begin{equation}\label{dtl}
\text{degree}(\widetilde{\mathbb L})\,=\, \text{degree}(f_k)\cdot \text{degree}({\mathbb L})+
4\frac{k-1}{2}\,=\, -k + 2k-2\,=\, k-2\,=\, \text{genus}(\Sigma_k)-1\,.
\end{equation}

Lemma \ref{lem:realSU4} then shows that the monodromy representation of  $(f_k^*\wt\nabla). h$ is conjugate to $\SL(2,\R)$ and its Euler class is maximal by \eqref{dtl}. 
More precisely, since the map $f_k$ in (\ref{fk}) is a k-fold covering of  ${\mathcal S}_4$ totally branched over the 4 marked 
points,  Lemma \ref{lem:realSU4} gives that the monodromy representation for 
$(f_k^*\wt\nabla). h$ coincides with the monodromy of the uniformizing hyperbolic structure for $\Sigma_{k}$. 
Therefore, the monodromy homomorphism of the connection $(f_k^*\wt\nabla). h$ coincides with the one given by the 
hyperbolic uniformization of $\Sigma_{k}$.

\subsection*{Case 2: $k$ is even}$\;$\\
Following the same desingularization procedure as in the previous case, consider the local gauge
transformation 
\begin{equation}\label{ke}
h_l\,=\,
\begin{pmatrix} y^{-\frac{k-1}{2}}&0\\0& y^{\frac{k-1}{2}}\end{pmatrix}
\end{equation}
with respect to the pull-back by $f_k$ of the $C^\infty$ decomposition of the rank 2 bundle in
\eqref{decomposeC2L}. As in the proof of point (2) of Proposition \ref{red},
the values of $y^{\frac{k-1}{2}}$ produce a ramified double covering of $\Sigma_k$
$$
\delta\, :\, \widetilde{\Sigma}\, \longrightarrow\, \Sigma_k
$$
which is ramified exactly over the subset $\{p_1,\, p_2,\, p_3,\, p_4\}$ in
\eqref{p4}. As before let
$$
\widetilde{\Sigma}'\,:=\, \delta^{-1}(\Sigma'_{k})\, \subset\, \widetilde{\Sigma}
$$
be the largest open subset such that $\delta\vert_{\widetilde{\Sigma}'}\, :\, \widetilde{\Sigma}'
\, \longrightarrow\, \Sigma'_{k}$ is an unramified double covering.
Let $q_l\,\in \, \widetilde{\Sigma}$, $1\, \leq\, l\, \leq\, 4$, be the points such that $\delta(q_l)\,=\,
p_l$. As in in the proof of part (1), fix a $C^\infty$ automorphism
$$
h\, \, :\, {\mathcal O}^{\oplus 2}_{\widetilde{\Sigma}'}\, \longrightarrow\,
{\mathcal O}^{\oplus 2}_{\widetilde{\Sigma}'}
$$
such that
\begin{itemize}
\item $\det h\,=\, 1$,

\item the action of $\text{Gal}(\delta)\,=\, {\mathbb Z}/2{\mathbb Z}$ on
$\delta^*{\mathcal O}^{\oplus 2}_{\Sigma'_k}\,=\,{\mathcal O}^{\oplus 2}_{\widetilde{\Sigma}'}$
takes $h$ to $-h$, and

\item the restriction of $h$ near each marked point $q_l$ coincides with
\[h_l\,=\,\begin{pmatrix} {\widetilde y}^{1-k}&0\\0& \widetilde{y}^{k-1}\end{pmatrix},\]
where ${\widetilde y}^{2k}\,=\, z\circ f_k\circ\delta$ with $z$ being a holomorphic coordinate
function around $f_k(p_l)\, \in\, \CP^1$ with $z(f_k(p_l))\,=\, 0$.
\end{itemize}

The $C^\infty$ connection $((\delta^*f^*_k \wt\nabla)\vert_{\widetilde{\Sigma}'}). h$ on
${\mathcal O}^{\oplus 2}_{\widetilde{\Sigma}'}$ (considered as  the trivial $C^\infty$ vector bundle)
extends to a flat connection on
$\delta^*{\mathcal O}^{\oplus 2}_{\Sigma_k}\,=\,{\mathcal O}^{\oplus 2}_{\widetilde{\Sigma}} $ preserved by the action of $\text{Gal}(\delta)$ on
$\delta^*{\mathcal O}^{\oplus 2}_{\Sigma_k}$. Hence it induces a $C^\infty$ flat connection on the trivial bundle ${\mathcal O}^{\oplus 2}_{\Sigma_k}$; this flat connection is $\widehat{\nabla}$ in the
statement of the proposition. As before we  emphasize that the  holomorphic structure given by
$\widehat{\nabla}$ does not coincide with the natural holomorphic structure of ${\mathcal O}^{\oplus 2}_{\Sigma_k}$
but gives a uniformization bundle.

In order to see how exactly $\widehat \nabla$ and $f_k^*\wt \nabla$ correspond to each other on $\Sigma_k$, we consider the holomorphic line bundle ${\mathcal L}\, \longrightarrow\, \Sigma_k$ equipped with the
logarithmic connection $\nabla^{\mathcal L}$ as in the proof of
point (2) in Proposition \ref{red}. It is straightforward to check that $\widehat{\nabla}$ and
$({\mathcal L},\, \nabla^{\mathcal L})$ satisfy all the properties stated in the proposition.
The homomorphism $\psi$ in the proposition is given by $h \otimes \mathbf 1$.
\end{proof}

\begin{remark}
The reason why we have to use a 2-valued gauge transformation for even $k$ (and hence the flat
line bundle $(\mathcal L,\, \nabla^{\mathcal L})$)
is that a hyperbolic isometric rotation
by an angle $\tfrac{2\pi}{\widetilde k}$
for $\widetilde k\,\in\, 2\Z$ cannot be represented by an $\SL(2,\R)$-matrix of order $\widetilde k$ but only by a
$\SL(2,\R)$-matrix of order $2\widetilde k$. See also
\cite[Section 4]{BoHS}  for the related case of symmetric minimal surfaces in 
${\mathbb S}^3.$
\end{remark}

\begin{proposition}\label{component-pull-back}
Let $k\,\in\, \N^{\geq3}$ and $\wt\rho\,=\, \frac{k-1}{2k}.$
Fix base points $p_0\, \in\, S_4$ and $p\,\in\, f^{-1}_k(p_0)\, \subset\, \Sigma_k\setminus\{p_1,\,p_2,\,p_3,\,p_4\}$.
Consider two logarithmic connections $D_1$ and $D_2$ on ${\mathcal O}^{\oplus 2}_{{\mathcal S}_4}$
such that
the two monodromy homomorphisms lie in the same connected component of ${\rm Hom}(\pi_1(S_4,\, p_0),
\, {\rm SL}(2,{\mathbb R}))$,  with the same prescribed local conjugacy classes determined by the parabolic weight $\wt\rho$.
Then the following hold:
\begin{enumerate}
\item If $k$ is odd, the pull-back through $f_k$ in (\ref{fk}) of the monodromy representations of $D_1$ and $D_2$ lie in the same connected
component of ${\rm Hom}(\pi_1(\Sigma_k,\,p),\, \SL(2,\R))$.

\item If $k$ is even, 
then the monodromy representations of $(f_k^*D_j)\otimes
\nabla^{\mathcal L}\,,$ $j\,=\,1,\,2$, lie in the same connected
component of  ${\rm Hom}(\pi_1(\Sigma_k,\,p),\, \SL(2,\R))$, where
 $\nabla^{\mathcal L}$ is the logarithmic connection defined in 
Proposition \ref{FRS}.
\end{enumerate}
\end{proposition}
\begin{proof}
We prove the statement only for odd k; the even case works analogously.
The (unbranched) covering $f_k\,\colon\, \Sigma'_k\,\longrightarrow\, S_4$
induces a covering-monodromy
$$\pi_1(S_4, \,p_0)\,\longrightarrow\,\mathcal S(k)$$ into the symmetric group 
$\mathcal S(k)\,\cong \,\mathcal S(f_k^{-1}(p_0)).$
Therefore, its first fundamental group $\pi_1(\Sigma_k\setminus\{p_1,\cdots ,p_4\},\,p)$ can be identified 
with the subgroup of $\pi_1(S_4,\, p_0)$ which is given by the kernel of the covering-monodromy. 
Moreover, the inclusion map $$\Sigma_k\setminus\{p_1,p_2,p_3,p_4\}\,\hookrightarrow\,\Sigma_k$$
induces a surjective homomorphism of fundamental groups  $$\pi_1(\Sigma_k\setminus\{p_1,p_2,p_3,p_4\},\,p)\,\longrightarrow\,
\pi_1(\Sigma_k,\,p).$$

The monodromy morphism commutes with the pull-back by $f_k$. Moreover, 
since $D_1$ and $D_2$ have real monodromy 
representation and parabolic weights $\tfrac{k-1}{2k}$,  the 
monodromy representation of the flat connection $f_k^*D_j$ on $\Sigma_k\setminus\{p_1,p_2,p_3,p_4\}$ factors through a representation 
of $\pi_1(\Sigma_k,\,p)$, for $j\,=\, 1,\,2$ (as the local monodromy at the marked points is trivial for both
the connections). By hypothesis, the monodromy homomorphisms for $D_1$ and $D_2$ are in the same connected component of
$\SL(2,\R)$-representations, and hence their monodromy representations
can be joined by a continuous path inside the space of  $\SL(2,\R)$-representations of  $\pi_1(S_4,\, p_0)$
with fixed local monodromies. The pull-back
of this path to the subspace of $\SL(2,\R)$-representations  
of  $\pi_1(\Sigma_k\setminus\{p_1,\cdots ,p_4\},\,p)$, lying in the kernel of the covering-monodromy,
is continuous as well. Moreover, by the same arguments as above, all these representations (determined by the path) factor through representations of
$\pi_1(\Sigma_k,\,p)$. Recall that $\pi_1(\Sigma_k\setminus\{p_1,p_2,p_3,p_4\},\,p)\,\longrightarrow\,
\pi_1(\Sigma_k,\,p)$ is surjective.
Therefore,  $f_k^* D_1$ and $f_k^* D_2$ are in the same connected  component of $\SL(2,\R)$-representations.
\end{proof}

\section{Logarithmic connections on the square torus with one marked point}

We consider the square torus
\begin{equation}\label{et}
T^2\,:= \,\C / \Gamma
\end{equation}
with lattice
$$\Gamma\,=\, {\mathbb Z}+\sqrt{-1}{\mathbb Z}\,\subset\,
\mathbb C$$ 
and one marked point $o\,=\,[0]\,\in\, T^2.$ The point $\frac{1+\sqrt{-1}}{4}\,\in\, T^2$ will be denoted by $p_0$.

Recall that the fundamental group $\pi_1(T^2\setminus\{o\},\, p_0)$ of the one-punctured torus $T^2 \setminus
\{o\}$ is a free group of two generators; it is generated by $\gamma_x,\,\gamma_y\,\in\,\pi_1(T^2\setminus\{o\}),$ 
where
\begin{equation}\label{gamma_x}
\gamma_x\,\colon\, [0,\, 1]\,\longrightarrow\, T^2\setminus\{o\};\ \, s\,\longmapsto\,
s+
\frac{1+\sqrt{-1}}{4}
\end{equation}
and
\[\gamma_y\,\colon\, [0,\,1]\,\longrightarrow\, T^2\setminus\{o\};\ \, s\,\longmapsto\,
\sqrt{-1}s+\frac{1+\sqrt{-1}}{4}.\]
The commutator $ \gamma_y^{-1}\gamma_x^{-1}  \gamma_y\gamma_x  \,\in\, \pi_1( T^2\setminus\{o\})$
 corresponds to a simple loop 
going around the marked point $o$. 

\subsection{The character variety of the one-punctured torus}\label{se6.1}$\;$\\
For $\rho\,\in\,]0,\, \tfrac{1}{2}[$, let $\mathcal M_{1,1}^\rho$ be the moduli space of flat
$\mathrm{SL}(2,\C)$-connections on the one-punctured torus $T^2 \setminus \{o\}$ (defined
in \eqref{et}) with local monodromy
around the puncture $o$ lying in the conjugacy class of the element 
\begin{equation}\label{locmon}
\dvector{ e^{-2\pi \sqrt{-1} \rho} &0 \\ 0& e^{2\pi\sqrt{-1} \rho}}\,\in\,
{\rm SL}(2,{\mathbb C})\, .
\end{equation}
The above de Rham moduli space $\mathcal M^\rho_{1,1}$ depends only on the topology 
of $T^2 \setminus \{o\}$; in particular, it does 
not depend on the complex structure of $T^2$. The conjugacy class of the element  in \eqref{locmon} is determined by its trace,
which is $2\cos(2\pi\rho)$; see \cite{Gold}.

For a flat $\mathrm{SL}(2,\C)$-connection $\nabla$ on $T^2 \setminus \{o\}$, let $X,\,Y$ denote its monodromies along 
$\gamma_x,\,\gamma_y\,\in \,\pi_1( T^2\setminus\{o\},\, p_0)$ (defined in \eqref{gamma_x}) respectively. Let
\begin{equation}\label{xyz}
x\,=\,\Tr(X),\quad y\,=\,\Tr(Y),\quad z\,=\,\Tr(YX).
\end{equation}
The moduli space $\mathcal M_{1,1}^\rho$ is diffeomorphic (via the monodromy mapping) to the character variety of 
the one-punctured torus for which the conjugacy class of the local monodromy at the puncture
is the one in \eqref{locmon}; this character variety is given by the equation
\begin{equation}\label{character-equation}
x^2+y^2+z^2-xyz-2-2\cos(2\pi\rho)\,=\,0\, ,
\end{equation}
where $x,\, y,\, z\, \in\, \mathbb C$.
Equivalently, for a fixed $\rho\,\in\, ]0,\, \tfrac{1}{2}[$, any triple $(x,\,y,\,z)\,\in\,\C^3$ satisfying 
\eqref{character-equation} determines, up to conjugacy, a unique representation of $\pi_1(T^2\setminus\{o\},\, p_0)$ 
into $\mathrm{SL}(2,\C)$ such that the local monodromy around the puncture is conjugate to (\ref{locmon}),
and $x$, $y$, $z$ are as in \eqref{xyz}; see \cite{Gold}. 
Note that the character variety is smooth for  $\rho\,\in\, ]0,\, \tfrac{1}{2}[$. The next lemma gives a characterization of the real points in this character variety.

\begin{lemma}\label{real}
Take $\Theta \, \in\, {\rm Hom}(\pi_1(T^2\setminus\{o\},\, p_0),\, {\rm SL}(2,{\mathbb C}))$, and denote
$X\,=\, \Theta(\gamma_x)$, $Y\,=\, \Theta(\gamma_y)$. Assume that 
$x\,=\,\Tr(X),$ $z_1\,=\,\Tr(YX)$ and $z_2\,=\,\Tr(Y^{-1}X)$ are real. Then either $y\,=\,\Tr(Y)\,\in
\,\mathbb R$ or $x\,=\,0.$ 
\end{lemma}

\begin{proof}
A short computation (see also \cite{Gold}) shows that up to conjugation we can choose
\begin{equation}\label{XYnormailisation}
X\,=\,\begin{pmatrix} x&1\\-1&0\end{pmatrix},\quad Y
\,=\,\begin{pmatrix} 0&-\zeta\\\zeta^{-1}&y\end{pmatrix}
\end{equation}
with
\[z_1\,=\,\zeta^{-1}+\zeta.\]
For given $x,\,y$ the traces $z_1$ and $z_2$ are solutions of the quadratic equation in \eqref{character-equation}.
Using \eqref{XYnormailisation} we compute that
\[ z_2\,=\,x y-z_1.\]
If $x,\,z_1,\,z_2 \,\in\,\R$, then either $y\,\in\, \R$ or $x\,=\,0.$ 
\end{proof}

The following theorem proved in \cite[Section 2.6 $\&$ Section 3.3]{Gold} describes the connected components of 
the real points in the character variety.

\begin{theorem}[Goldman \cite{Gold}]\label{goldman}
For $\rho\,\in\,]0,\,\tfrac{1}{2}[$, the set of real points of the character variety defined by \eqref{character-equation} has
5 connected components. There is one compact component which is characterized by $x,\,y,\,z\,\in\,[-2,\,2]$, and
there are 4 non-compact 
components which are all diffeomorphic to each other. The compact component consists of $\mathrm{SU}(2)$-representations and the non-compact 
components consist of $\mathrm{SL}(2,\R)$-representations.
\end{theorem}

\begin{remark}\label{4realcomponents}
The four non-compact components of the character variety are interchanged by
the group of sign-change automorphisms \cite[Section 2.2.1 $\&$ Section 2.6]{Gold}. This means, that the
coordinates $(x,\,y,\,z)$ are mapped to $((-1)^{\epsilon_1}x,\,(-1)^{\epsilon_2}y,\,(-1)^{\epsilon_3}z)$
where $\epsilon_l\,\in\,\{0,1\}$ for $l\,=\,1,\,2,\,3,$ such that $\epsilon_1+\epsilon_2+\epsilon_3\in\{0,2\}.$
In terms of the Hitchin-Kobayashi correspondence, these four components
correspond to the four distinct spin structures on a torus.
\end{remark}

\subsection{The de Rham moduli space of the one-punctured torus}$\;$\\
Let $w$ be the global coordinate on the universal covering $\C$ of $T^2$ in \eqref{et}.
Since $T^2$ is a square torus, there exists an anti-holomorphic involution
\begin{equation}\label{et2}
\eta\,\colon\, T^2\,\longrightarrow\, T^2, \quad [w]\,\longmapsto\, [-\sqrt{-1}\overline{w}]
\end{equation}
on $T^2$ corresponding to the reflection along a diagonal of the square. 
Note that the marked point $o\in T^2$ is fixed by the map $\eta.$
The induced real involution of the de Rham moduli space
\[\mathcal M_{1,1}^\rho\,\longrightarrow\, \mathcal M_{1,1}^\rho ,\quad [\nabla]\,\longmapsto\,
[\eta^*\overline\nabla]\]
is well-defined as $\rho$ is real.

For notational convenience we denote by $L$ the trivial $C^\infty$ bundle $T^2\times {\mathbb C}\, \longrightarrow\, T^2$. Let $a,\,\chi \in \C$ be coordinates of $\mathcal M_{1,1}^\rho$ obtained from abelianization
(see \cite[(2.3)]{BDH}, or \cite{HeHe}). For this purpose recall from \cite[(2.3)]{BDH} that  any element in $\mathcal M_{1,1}^\rho$ (with  $\rho\,\in\,]0,\,\tfrac{1}{2}[$)  is represented by 
a logarithmic  flat connection on $L\oplus L^*$ with a unique pole at $o$  
\begin{equation}\label{abel-connection}
\nabla\,=\,\nabla^{a,\chi,\rho}\,=\,\dvector{\nabla^L &\gamma^-_\chi\\ \gamma^+_\chi &
\nabla^{L^*} }\,,
\end{equation}
where $\nabla^L$ is the flat connection on $L$ defined by
\begin{equation}\label{nablaL}
\nabla^L\,=\,d+adw+\chi d\overline{w};
\end{equation}
$w$ being the above global  holomorphic  coordinate of $T^2$ and $a,\,\chi \in \C$. Moreover $ \nabla^{L^*}$ is its dual connection on $L^*$, while 
$\gamma^+_\chi$ and $\gamma^-_\chi$ are meromorphic sections with respect to the holomorphic structure
given by the Dolbeault operators $\overline{\partial}^0 - 2 \chi d\overline{w}$ and  $\overline{\partial}^0 + 2 \chi d\overline{w}$ respectively,  with simple poles 
at $o\,\in\, T^2$ and residues  determined by $\rho$. Here $\overline{\partial}^0\,=\,d''$ is the $(0,1)$-part of
the de Rham differential operator  $d$; in particular, there is a  holomorphic structure induced by $\nabla$ in  (\ref{abel-connection}) on $L$, the one given  given by the Dolbeault operator $\overline{\partial}^0 +  \chi d\overline{w}$.

\begin{remark}
Note that the parabolic weight  at $o$  of the logarithmic connection   $\nabla^{a,\chi,\rho}$ 
is  $\rho.$ The parabolic line is determined,
up to a holomorphic automorphism of $L\oplus L^*$, by the
condition that it is neither the line $L_o$, nor the line $L^*_o.$
\end{remark}

\begin{lemma}\label{tausymcon}
Let $\Gamma^*\,:=\,\pi\Z+\sqrt{-1}\pi\Z.$ The gauge class of the connection $\nabla\,=\,\nabla^{a,\chi,\rho}$ as in (\ref{abel-connection})  on the one-punctured torus is fixed by the involution $\eta$ defined in (\ref{et2}) if
\[\chi\,\in\, (1-\sqrt{-1})\R\setminus\tfrac{1}{2}\Gamma^*\;\quad { and }\quad a\,\in\, (1+\sqrt{-1})\R ,\]
or
\[\chi\,\in\, (1+\sqrt{-1})\R\setminus\tfrac{1}{2}\Gamma^*\;\quad { and }\quad a\,\in\, (1-\sqrt{-1})\R.\]
\end{lemma}

\begin{proof}
We have
\[\eta^* dw\,=\,-\sqrt{-1} d\overline{w}\quad \text{ and }\quad \eta^* d\overline{w}\,=\,\sqrt{-1} d w.\]
Hence, for $\chi\,\in\, (1-\sqrt{-1})\R$ and $a\,\in\, (1+\sqrt{-1})\R$,
\[\eta^*\overline{\nabla^L}=\nabla^L\]
with
$\nabla^L$ given by \eqref{nablaL}. By \cite[Proposition 2.5]{BDH} the meromorphic  sections 
$\gamma^\pm_\chi$ in (\ref{abel-connection}), described above, are unique, up to scaling,  under  the  given condition that
the quadratic residue at $o\,\in\, T^2$ of the meromorphic quadratic differential 
\[\gamma^+_\chi\gamma^-_\chi (dw)^2\]
 is $\rho^2$. Thus, we obtain constants $c^+,c^-\,\in\,\C^*$, with $c^+ c^-\,=\,1$,
such that
\[\eta^*\overline{\gamma^\pm_\chi dw}\,=\,c^\pm\gamma^\pm_\chi dw.\]
In particular,
$\nabla$ and $\eta^*\overline\nabla$ are gauge equivalent.
If $\chi\,\in\, -(1-\sqrt{-1})\R\setminus\tfrac{1}{2}\Gamma^*$ and  $a\,\in\, -(1+\sqrt{-1})\R $
then
$\eta^*\overline{\nabla^L}=(\nabla^L)^*,$
and the proof works analogously.
\end{proof}

\begin{lemma}\label{tausymcon2}
Let $\nabla\,=\,\nabla^{a,\chi,\rho}$ be a connection on $T^2\setminus\{o\}$ as in (\ref{abel-connection}) with $[\eta^*\overline{\nabla}]
\,=\,[\nabla].$ Then
$$z_1\,=\,\Tr(YX)\,\in\, \R \ \quad\text{ and } \quad\ z_2\,=\,\Tr(Y^{-1}X)\,\in\,\R.$$
\end{lemma}

\begin{proof} Consider the $p_0$-based loops  $\gamma^{z_1}$ and $\gamma^{z_2}$  on $T^2\setminus\{o\}$ which are the concatenations of the loops $\gamma_x$ and $\gamma_y$ (defined in \ref{gamma_x}) and 
of the loops $\gamma_x$ and ${\gamma_y}^{-1}$ respectively. Their corresponding  elements   in $\pi_1( T^2\setminus\{o\},\, p_0)$ (for which we use the same notation) satisfy $\gamma^{z_1}=\gamma_y \gamma_x$ and
$\gamma^{z_2} = {\gamma_y}^{-1} \gamma_x$.
By definition $z_1$ and $z_2$ are the traces of the monodromy of $\nabla$ along the loops $\gamma^{z_1}$ and $\gamma^{z_2}$ respectively.

Note that the  real involution $\eta$ in \eqref{et2} maps the 
closed curve $\gamma^{z_2}$ to a curve $\eta(\gamma^{z_2})$ which is free homotopic (i.e., without fixed base point) 
 to $\gamma^{z_2}$; see Figure \ref{gammaz}. Since by hypothesis $\eta^*\overline{\nabla}\,\cong\,\nabla$, we thus obtain that $z_2\,=\,\overline{z}_2.$ Similarly, the closed curve 
$\gamma^{z_1}$  is mapped by 
$\eta$ to a curve $\eta(\gamma^{z_1})$ which is free homotopic to $(\gamma^{z_1})^{-1}$; see Figure \ref{gammaz}.
As $\Tr(M)\,=\,\Tr(M^{-1})$ for every $M\,\in\,\mathrm{SL}(2,\C)$, we obtain
that $z_1\,=\,\overline{z}_1.$
\end{proof}
\begin{figure}[ht]
\centering
\includegraphics[width=0.45\textwidth]{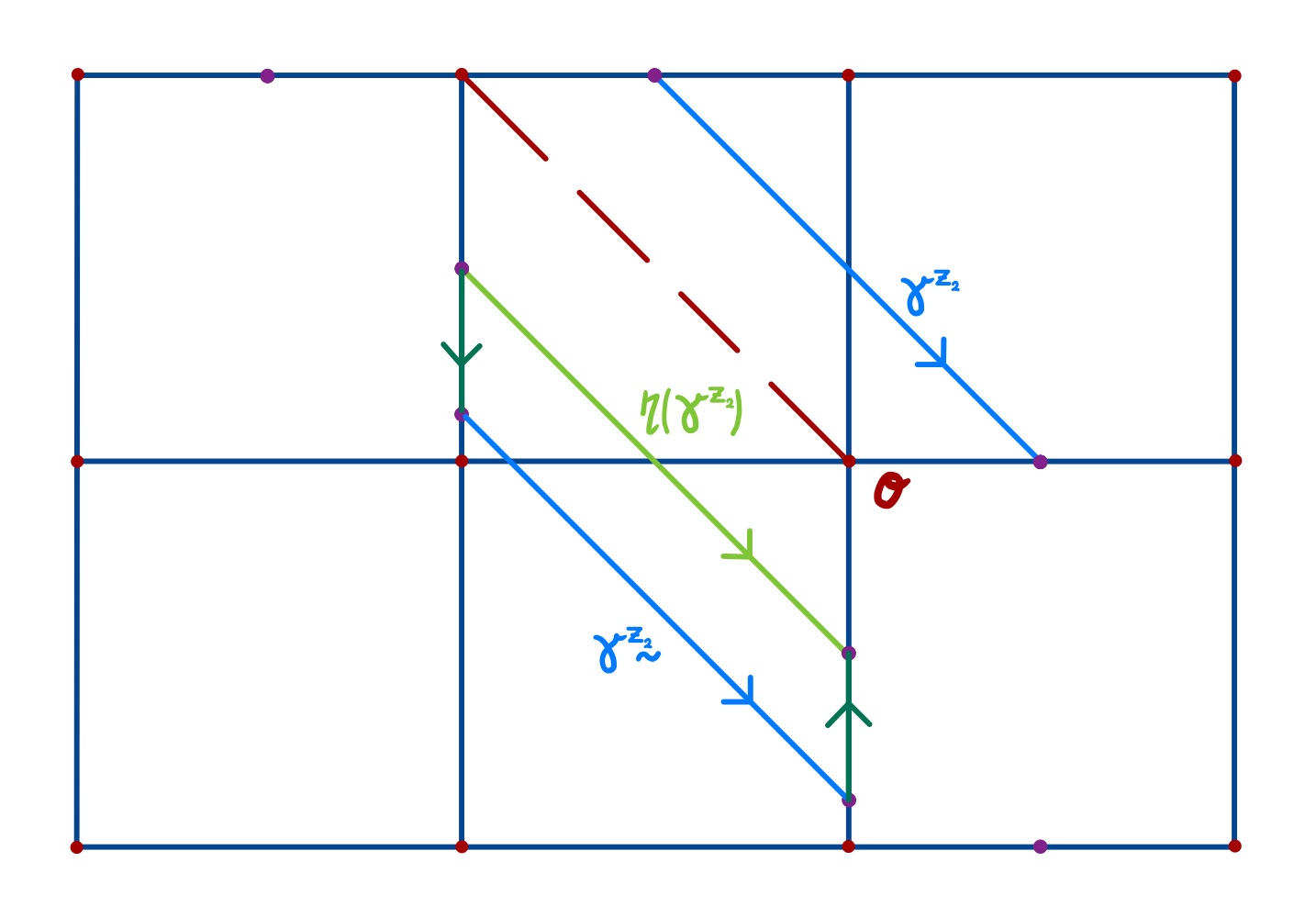}
\includegraphics[width=0.45\textwidth]{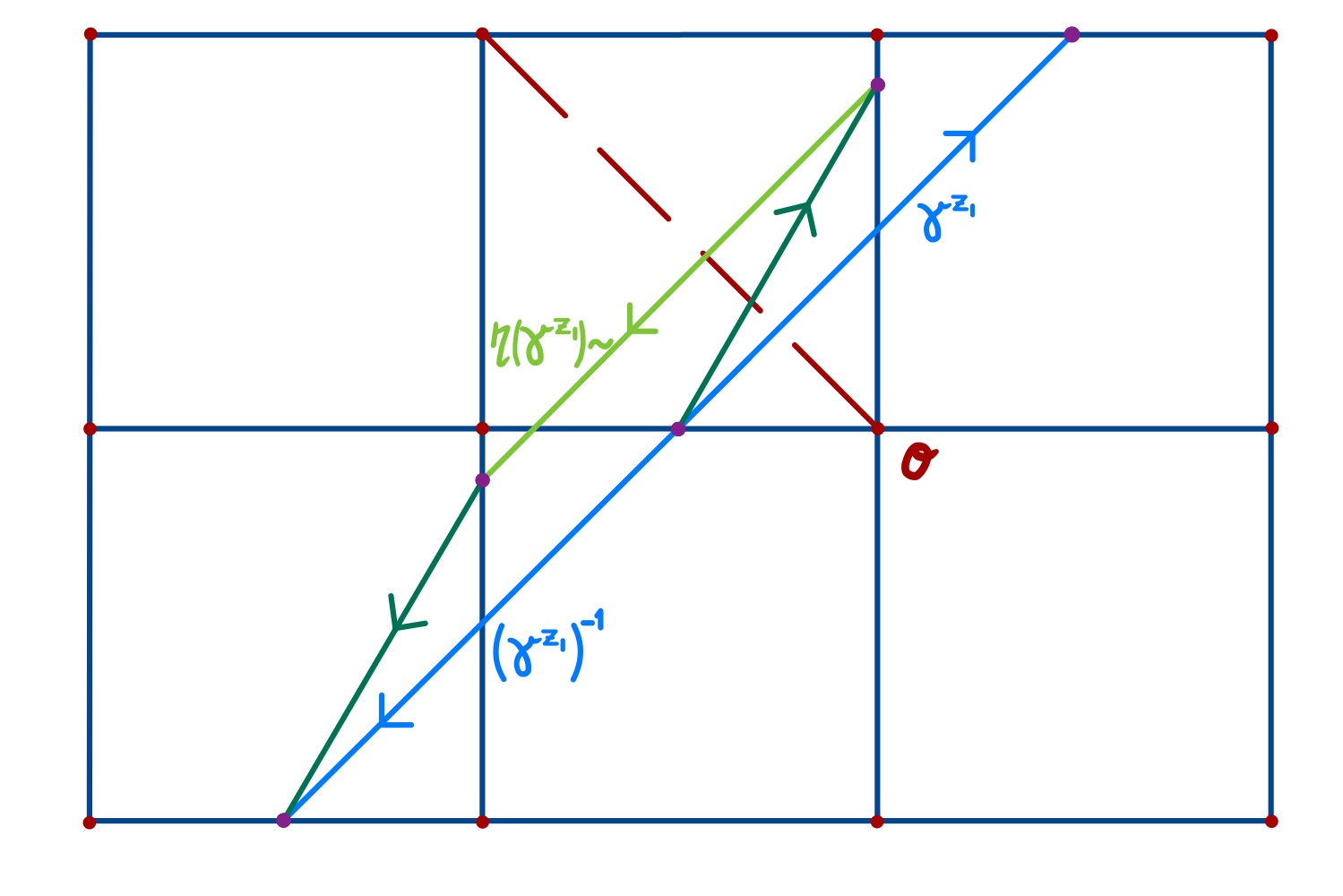}
\caption{The curves $\gamma^{z_2}$ and $\gamma^{z_1}$:  their  monodromy  traces are $z_2$ and $z_1$.}
\label{gammaz}
\end{figure}

\subsection{A consequence of WKB analysis}$\;$\\
Fix $\rho\,\in\,]0,\,\tfrac{1}{2}[$ and
\begin{equation}\label{eq:chi0}
\chi^0\,=\,\tfrac{\pi}{4}(1-\sqrt{-1}), \quad a^0\,=\,\tfrac{\pi}{4}(1+\sqrt{-1}).
\end{equation}
Consider the family of 
flat connections, parametrized by $t \,\in \,\R$ on $T^2\setminus\{o\}$ (defined in \eqref{et}) given by
\[\nabla^t\,:=\,\nabla^{ (1-t)\, a^0,\chi^0,\,\rho}\,=\,\nabla^{a^0,\chi^0,\,\rho}+ t\tfrac{\pi}{4}(1+\sqrt{-1})\begin{pmatrix} -dw&0
\\0&dw\end{pmatrix}.\]
In this section we study the behavior of \[x(t)\,:=\, \Tr(X(t)),\] where $X(t)$ is the monodromy of $\nabla^t$
along $\gamma_x\,\in\,\pi_1(T^2\setminus\{o\},\, p_0).$ By Lemma \ref{tausymcon},
the connection $\nabla^t$ is compatible with the involution $\eta$ (see \eqref{et2}),
in the sense that $[\nabla^t]\,= \,\eta^*[\overline{\nabla^t}]$ for all $t\,\in\,\R$.
In particular, the traces $z_1(t),\,z_2(t)$ defined in Lemma \ref{tausymcon2}
 are real for all  $\nabla^t$, with  $t\,\in\,\R.$

From the definition of $\gamma_x$ in \eqref{gamma_x} we have
\[\gamma_x'(s)\,=\,1 \quad \,\forall\,\ s\,\in\,[0,\,1].\]
For the vector $v\,=\,1\,\in\,\C$ we have
\[\text{Re}(-\tfrac{\pi}{4}(1+\sqrt{-1}) dw(v))\,=\,-1\,>\,0.\]
Hence, the curve $\gamma_x$
is a WKB curve (see Section \ref{subsection;21.3.7.30} in the Appendix)
for the 1-form
\[-\tfrac{\pi}{4}(1+\sqrt{-1}) dw.\] 
From Corollary \ref{cor;21.3.7.10} of the Appendix (compare also with \cite[Appendix 4]{GMN})
we get a non-zero constant $C\,\in\,\C^*$ such that
\begin{equation}\label{phaseanalysis}\lim_{t\in\R^{>0},\,t\to\infty}x(t)\exp(-t\pi\frac{1+\sqrt{-1}}{4})\,=\,C.
\end{equation}
From this the following corollary is obtained.

\begin{corollary}\label{tn}
There exist a sequence $(t_n)_{n\in \N}\,\subset\,\R$ such that $x(t_n)$ is real and non-zero for every $n
\,\in\, \N$, and ${\displaystyle \lim_{n\rightarrow \infty}}t_n\,=\,\infty$. In particular, the monodromy
representation of $\nabla^{t_n}$ is conjugate to a $\mathrm{SL}(2,\R)$-representation for all $n$.
\end{corollary}

\begin{proof}
Equation \eqref{phaseanalysis} with $C\,\neq\,0$ yields a sequence $(t_n)_{n \in \N}\,\subset\, \R$ such that
\[x(t_n)\,\in\,\R\setminus[-2,\,2]\] for all $n \,\in\, \N$. From
Lemma \ref{tausymcon2} we know that $z_1\,=\,z_1(t_n)$ and $z_2\,=\,z_2(t_n)$ are both real. Recall that 
\[z_2\,=\,xy-z_1\] and $x(t_n)\,\neq\,0.$
Therefore, Lemma \ref{real} shows that $y(t_n)\,\in\,\R$ for all $n$, and hence the representation
is given by a real point in the character variety. Since $x(t_n)\,\in\,\R\setminus[-2,\,2] $, Goldman's result
(Theorem \ref{goldman}) implies that the monodromy representation
of $\nabla^{t_n}$ is conjugated to an $\mathrm{SL}(2,\R)$ representation for all $n\,\in\,\N.$
\end{proof}

Corollary \ref{tn} shows the existence of logarithmic connections $\nabla^{t_n}$ on the
one-punctured torus $T^2$  with real monodromy. Recall that  the Dolbeault operator  $\overline{\partial}^0 +  \chi d\overline{w}$  is gauge equivalent with $\overline{\partial}^0 $ (and hence defines the trivial holomorphic line bundle structure on $L$) if and only if 
$2 \chi \in \Gamma^*$, with $\Gamma^*$ defined in Lemma \ref{tausymcon}. Hence, it follows that the holomorphic structure of $L$ induced by  $\nabla^{t_n}$ (for which $\chi^0\,=\,\tfrac{\pi}{4}(1-\sqrt{-1})$) is that of a holomorphic  line bundle of order 4 on $T^2$. In order to lift $\nabla^{t_n}$, for an appropriate $\rho,$ to the 
Riemann surface $\Sigma_k$, we first need to relate
the moduli space ${\mathcal M}^\rho_{1,1}$ in Section \ref{se6.1} with the 
moduli space of flat connections ${\mathcal M}^{\wt\rho}_{0,4}$ on $\mathcal S_4.$

\subsection{Abelianization and connection}\label{ss:abel}$\;$\\
In \cite{HeHe}, logarithmic $\mathfrak{sl}(2,\C)$-connections $d+\xi$ on the
rank two trivial 
holomorphic bundle on $\CP^1$ with four marked points $\{\pm 1,\, \pm \sqrt{-1})$ are studied by an 
abelianization procedure. We need to recall (and adapt to our
situation) some of the results of \cite{HeHe}. We restrict hereby to 
logarithmic connection on ${\mathcal O}_{{\mathcal S}_4}^{\oplus 2}$
such that all residues have the same eigenvalues
\begin{equation}\label{eq:evfuchsian}
\pm\wt\rho\quad\
\text{for\,~ some}\,\quad\, \wt\rho\,\in\,]\tfrac{1}{4},\,\tfrac{1}{2}[.
\end{equation}

\subsubsection{The character variety of a four-punctured sphere}

As before $S_4$ denotes  the complex projective line $\CP^1$  with punctures at the points  \begin{equation}\label{xl}
x_l\,:=\,e^{(l-1)\sqrt{-1}\frac{\pi }{2}}
\end{equation}
 for $l\,=\,1,\, \cdots,\,4$ and  $p_0 \in S_4$ a base point. For any $l\,=\,1,\, \cdots,\,4$, consider a simple oriented  $p_0$-based loop $\gamma_{x_l}$  going around the puncture $x_l$. The fundamental group $\pi_1(S_4, \,p_0)$ is generated by $\gamma_{x_l}$, with $l\,=\,1,\, \cdots,\,4$; the generators  satisfy  the relation
 $\gamma_{x_4}\gamma_{x_3}\gamma_{x_2}\gamma_{x_1}\,=\,\text{Id}.$ The following is a well-known result dating back to Fricke; see \cite{Gold88}.

Any  $\SL(2,\C)$-representation of    $\pi_1(S_4, \,p_0)$ is   determined by the images 
$M_l\,\in\,\SL(2,\C)$ of the generators  $\gamma_{x_l} \in \pi_1(S_4, \,p_0)$, for $l\,=\,1,\, \cdots,\,4$. We have
\[M_4M_3M_2M_1\,=\,\text{Id}.\]
 Let 
\[\mu=2\cos(2\pi\wt \rho).\]
We restrict to the case 
\[\Tr(M_l)\,=\,\mu\;\quad\forall\,\,\, \;l\,=\,1,\,\cdots,\,4.\]
If the representation is irreducible or totally reducible, the traces
\[\wt x\,=\,\Tr(M_2M_1),\,\ \wt y\,=\,\Tr(M_3M_2),\,\ \wt z=\Tr(M_3M_1)\]
determine the representation uniquely up to conjugation. Moreover, these affine coordinates
$(\wt{x},\,\wt{y},\,\wt{z})$ satisfy the equation
 \begin{equation}\label{Fricke4}
 \wt x^2+
\wt y^2 +\wt z^2 +\wt x\wt y\wt z-2\mu^2 (\wt x+\wt y+\wt z)+4(\mu^2 -1)+\mu^4 \,=\,0.
 \end{equation}
Furthermore, a totally reducible representation is conjugate to a $\mathrm{SU}(2)$-representation if and only if $\wt x,\wt y,\wt z\in[-2,2]$, while
it is conjugate to an $\SL(2,\R)$-representation if $\wt{x},\,\wt{y},\,\wt{z}\,\in\,\R$ but 
not \[\wt{x}\,\in\,[-2,\,2]\,\quad\text{ and }\,\quad\wt{y}\in[-2,\,2]\,\quad\,\text{ and }\,\quad\, \wt{z}
\,\in\,[-2,\,2].\]

\subsubsection{Abelianization}\label{sssAbel}

We consider logarithmic connections $d+\xi$ on the rank two  trivial holomorphic  bundle over $\mathcal S_4$ which are symmetric, in the sense 
that all four residues have eigenvalues $\pm\wt\rho$. As explained in Section \ref{ss:parabolic}, a logarithmic 
connection induces a parabolic bundle $E$. The parabolic weights are hereby
$\wt\rho$ at each of the four singular points. The generic underlying holomorphic vector bundle for parabolic
bundles is trivial. So once the parabolic weight is fixed, the parabolic structure $E$ is 
essentially determined by the lines defining the quasiparabolic structures, or in other words,
the cross-ratio of the 4 quasiparabolic lines in the trivial vector space $\C^2$; see \cite{LoSa} or \cite{HeHe}.

It can be shown (see \cite[Proposition 2.1]{HeHe}) that for a  generic parabolic structure $E$, i.e., for a generic 
cross-ratio of the 4 parabolic lines, the space of strongly parabolic Higgs fields is complex $1$-dimensional. 
Moreover, for a generic parabolic structure $E$, the determinant of a non-zero strongly parabolic Higgs field 
$\theta$ is a non-zero constant multiple of
\begin{equation}\label{ed}
\frac{(dz)^2}{z^4-1}.
\end{equation}
Take a strongly parabolic Higgs bundle $(E,\, \theta)$ such that $\det\theta$ is non-zero constant
multiple of \eqref{ed}. Let
\[f\, :\, \Sigma_2\,\longrightarrow\, \CP^1\]
be the spectral curve and ${\mathbf L}\,\longrightarrow\, \Sigma_2$ the holomorphic line bundle corresponding to 
$(E,\, \theta);$ see also \cite{Hi1} for the smooth case. We recall that $\Sigma_2$ is contained in the total space of $K_{\CP^1}\otimes {\mathcal 
O}_{\CP^1}(x_1+x_2 +x_3+ x_4)$, where the $x_l$'s are the fourth roots of unity as in \eqref{xl}, and $f\,=\,f_2$  (as in (\ref{fk}) for $k=2$)
is the ramified double cover of $\CP^1$ branched over the singular points $x_1,\, x_2,\, x_3,\, x_4$; the 
holomorphic line bundle ${\mathbf L}$ is the subbundle of $f^*E\,=\, {\mathcal O}^{\oplus 2}_{\Sigma_2}$ given by 
the eigenline bundle of $\theta$. We have ${\rm genus}(\Sigma_2)\,=\, 1$ and $\text{degree}({\mathbf L})\,=\, 
-2$. As before, denote the point $f^{-1}(x_l)$ by $p_l$. Let $\sigma\, :\, \Sigma_2\,\longrightarrow\, \Sigma_2$ 
be the nontrivial element of the Galois group $\text{Gal}(f)$. Then \[{\mathbf L}\otimes\sigma^* {\mathbf 
L}\,=\,{\mathcal O}_{\Sigma_2}(-p_1-p_2-p_3-p_4);\] see \cite[Section 3]{HeHe}. When $p_1$ is chosen as the 
identity element of the addition law, $p_2,\, p_3,\, p_4$ become the nontrivial order two points of the elliptic 
curve. So
\[-3p_1+p_2+p_3+p_4\]
is a principal divisor (associated to the derivative  $\wp'$ of  the Weierstrass $\wp$-function), and therefore
\[({\mathcal O}_{\Sigma_2}(-2p_1))^{\otimes 2}\,=\, {\mathbf L}\otimes\sigma^* {\mathbf L}.\]
Thus, there is $L_0\,\in\,\text{Jac}(\Sigma_2)$ 
with
\begin{equation}\label{defL0}
{\mathcal O}_{\Sigma_2}(-2p_1) \otimes L_0\,=\, {\mathbf L}\ \ \text{ and }\ \
{\mathcal O}_{\Sigma_2}(-2p_1) \otimes L^*_0\,=\, \sigma^*{\mathbf L}
.\end{equation}
Consider the logarithmic connection on ${\mathcal O}_{\Sigma_2}(-p_1-p_2-p_3-p_4)$ given by the de Rham differential.
It produces a logarithmic connection on ${\mathcal O}_{\Sigma_2}(-4p_1)$ once an 
isomorphism of ${\mathcal O}_{\Sigma_2}(-p_1-p_2-p_3-p_4)$ with ${\mathcal O}_{\Sigma_2}(-4p_1)$ is chosen (for instance,  the isomorphism defined by the multiplication with  $\wp'$); this
connection on ${\mathcal O}_{\Sigma_2}(-4p_1)$ does not depend on the choice of the isomorphism. A connection on
${\mathcal O}_{\Sigma_2}(-4p_1)$ produces a connection on ${\mathcal O}_{\Sigma_2}(-2p_1)$. Let
\begin{equation}\label{Ds2}D^s\end{equation}
be the logarithmic connection on ${\mathcal O}_{\Sigma_2}(-2p_1)$ obtained this way. 
It satisfies the equation
\begin{equation}\label{eqDs}
D^s s_{-2 p_1}\,=\,-\frac{d\wp'}{2\wp'}\otimes s_{-2 p_1},\end{equation}
where $s_{-2 p_1}$ is the meromorphic section with double pole at $p_1$.

In particular, $D^s$ is singular at
$p_1,\,\cdots,\, p_4$,   all  residues being equal to   $\frac{1}{2}$ (and hence the monodromy around the singular points being  $-1$); for more 
details see  (the proof of) \cite[Theorem 3.2]{HHSch}. Denote by $(D^s)^*$  the dual connection of $D^s$
on $\mathcal O_{\Sigma_2}(2p_1)$.

The holomorphic bundle underlying the pull-back $f^*E$ of the parabolic bundle $E$ is the rank two  trivial holomorphic 
 bundle over $\Sigma_2.$ Recall that both ${\mathbf L}$ and $\sigma^*{\mathbf L}$ are holomorphic subbundles 
of the rank two trivial holomorphic bundle over $\Sigma_2$. This inclusion map defines a holomorphic vector 
bundle map
\[ {\mathbf L}\oplus \sigma^*{\mathbf L}\to\mathcal O_{\Sigma_2}\oplus \mathcal O_{\Sigma_2}\]
which is an isomorphism away from the divisor $p_1+p_2+p_3+p_4.$ Consider now the holomorphic isomorphism
\[L_0\oplus L_0^*\to ({\mathbf L}\oplus \sigma^*{\mathbf L})\otimes \mathcal O_{\Sigma_2}(2p_1).\]
It is shown in \cite[Section 3]{HeHe}
that the induced  logarithmic connection $( f_2^*(d+\xi))\otimes (D^s)^*$ on \[L_0\oplus L_0^*\]
 is given by
\[d+\begin{pmatrix} \nabla^{L_0}&\beta^-\\ \beta^+& (\nabla^{L_0})^*\end{pmatrix}.\]
Here, $\nabla^{L_0}$ and $(\nabla^{L_0})^*$ are dual holomorphic line bundle connections on $L_0$ respectively $L_0^*.$
Moreover, 
the second fundamental forms $\beta^+$ and $\beta^-$ are meromorphic sections of
\[K_{\Sigma_2}\otimes L_0^{-2}\ \ \text{ and } K_{\Sigma_2}\otimes L_0^{2} \]
respectively; they
can be explicitly determined in terms of $\vartheta$-functions \cite[Proposition 3.2]{HeHe}. 
Moreover, the eigenvalues of the residues of $( f_2^*(d+\xi))\otimes(D^s)^*$
 are
\[\pm(2\wt\rho-\tfrac{1}{2})\]
which implies that the quadratic residue of the meromorphic quadratic differential  $\beta^+\beta^-$ is $(2\wt\rho-\tfrac{1}{2})^2.$

The relationship between the abelianization of symmetric logarithmic connections on ${\mathcal S}_4$
and flat connections on the one-punctured torus is given as follows.
Consider the 4-fold covering induced by the identity map on $\C$
\[\pi_4\,\colon\,\Sigma_2\,=\,\C/(2\Z+2\sqrt{-1}\Z)\,\longrightarrow\, T^2\,=\,\C/(\Z+\sqrt{-1}\Z).\]
The pull-back of topologically trivial holomorphic line bundles defines a 4-fold covering
\[\text{Jac}(T^2)\,\longrightarrow\, \text{Jac}(\Sigma_2).\]
Spin bundles on $T_2$ are mapped to the trivial holomorphic line bundle on $\Sigma_2.$ Further,  holomorphic line bundles of order 4 on $T^2$ are mapped to  nontrivial spin bundles on  $\Sigma_2.$

As shown in \cite[Section 3.1]{HHSch} (see also \cite[Remark 3.3]{HeHe} and \cite[Section 4]{He}),
for a symmetric  logarithmic connection $d+\xi$ with local weights $\wt\rho$
on $\mathcal S_4$ with underlying parabolic bundle admitting a strongly 
parabolic Higgs field of non-vanishing determinant, there
exists $a,\,\chi\,\in\,\C,$\, $\chi\,\notin\,\tfrac{1}{2}\Gamma^*$, such that
$f_2^*(d+\xi)\otimes(D^s)^*$ and $\pi_4^*\nabla^{a,\chi,\rho}$ are gauge
equivalent (with the connection $\nabla^{a,\chi,\rho}$ as in (\ref{abel-connection})) and
\[\pi_4^*\gamma^\pm\,=\,\beta^\pm.\]
The  above abelianization-procedure  leads to the following theorem.

\begin{theorem}\label{thm:4:1}
Let $\rho\,\in\,]0,\,\tfrac{1}{2}[$ and $\wt\rho\,=\,\frac{2\rho+1}{4}.$
There is a degree 4 birational map
\[\mathcal M_{1,1}^\rho\,\longrightarrow\, \mathcal M_{0,4}^{\wt\rho}\]
compatible with the underlying parabolic structures. 
On the character variety this map is given by
\[(x,\,y,\,z)\,\longmapsto\, (\wt x,\,\wt y,\,\wt z)\,=\,(2-x^2,\,2-y^2,\,2-z^2).\]
\end{theorem}

\begin{remark}\label{rem-stability}
In our symmetric  case, where  the parabolic weight $\wt\rho\in ]\tfrac{1}{4},\,\tfrac{1}{2}[ $ 
is the same at every marked point of $\mathcal S_4,$ there are only two 
polystable parabolic structures which admit a compatible logarithmic connection (as defined in  Section \ref{ss:parabolic}), but no strongly  parabolic Higgs field with 
non-zero determinant. The first of the two exceptions is induced by $\wt\nabla$ constructed in \eqref{wtnabla}, 
and the second is a stable parabolic structure defined on ${\mathcal O}_{\mathcal S_4}(1)\oplus{\mathcal 
O}_{\mathcal S_4}(-1).$

There are exactly three totally reducible connections having semistable parabolic structure on $\mathcal S_4,$ 
(see \cite{HeHe}); one of them being $D$ as defined in \eqref{eq:D}. 
The semistable parabolic structures of
 these three totally reducible connections admit 
strongly parabolic Higgs fields with non-zero determinant. In particular, the parabolic bundle induced by $D$ has 
the strongly parabolic Higgs field defined in \eqref{Phi}.
The corresponding line bundles $L_0$ in 
\eqref{defL0} of these semistable parabolic structures are exactly the non-trivial spin bundles of $\Sigma_2$. They correspond to holomorphic line bundles 
of order 4 on $T^2$. 
\end{remark}

\begin{proof}[{Proof of Theorem \ref{thm:4:1}}]
The birational map $\mathcal M_{1,1}^\rho\,\longrightarrow\,\mathcal M_{0,4}^{\wt\rho}$ is given via
abelianization.  Note that 
for a  nontrivial  Zariski  open set  in 
$M_{0,4}^{\wt\rho}$, the parabolic bundle  induced via
the  Riemann-Hilbert correspondence is defined on the 
rank two trivial holomorphic bundle over ${\mathcal S}_4$ and
admits a
parabolic Higgs field of non-vanishing determinant,
see for example \cite{LoSa}.
As explained above, there
exists $a,\,\chi\,\in\,\C,$\, $\chi\,\notin\,\tfrac{1}{2}\Gamma^*$, such that
$f_2^*(d+\xi)\otimes(D^s)^*$ and $\pi_4^*\nabla^{a,\chi,\rho}$ are gauge
equivalent by \cite{HeHe}. The connection $\nabla^{a,\chi,\rho}$  is a preimage of $d + \xi$ through  our birational map.

There are four preimages, because  the pull-backs of two connections
$\pi_4^*\nabla^{a_1,\chi_1,\rho}$ and $\pi_4^*\nabla^{a_2,\chi_2,\rho}$ from the one-punctured to the four-punctured
torus are gauge equivalent if and only if 
they differ by a spin-connection, i.e.,
\[(a_2-a_1,\,\chi_2-\chi_1)\,=\,(-\overline{\nu},\,\,\nu), \quad\ \nu\,\in\,\tfrac{1}{2}\Gamma^*.\]
Recall also that the elements in $\mathcal M_{1,1}^\rho$ admitting a representative of the form $\nabla^{a,\chi,\rho},$  with $a,\,\chi\,\in\,\C,$\, $\chi\,\notin\,\tfrac{1}{2}\Gamma^*$, form a nontrivial Zariski open set  (see \cite[Theorem 1]{HeHe} or
\cite[Section 2.3]{BDH}).

It remains to determine the relationship between the character varieties. First observe that Equation
\eqref{Fricke4} for $(\wt x, \,\wt y,\, \wt z)\,=\, (2-x^2,\, 2-y^2,\, 2-z^2)$ factors as
$$
(x^2+y^2+z^2-xyz-4+\mu^2)(x^2+y^2+z^2+xyz-4+\mu^2)\,=\,0
$$
with $\mu\,=\,2\cos(2\pi\wt\rho).$ Replacing \[2-\mu^2 \,= \,\kappa \,:= \,2\cos(2\pi\rho)\] then gives
$$
(x^2+y^2+z^2-xyz-2+\kappa)(x^2+y^2+z^2+xyz-2+\kappa)\,=\,0.
$$
The first factor coincides with Equation \eqref{character-equation} for the one-punctured torus with parabolic 
weight $\rho$. Hence, the map between the character varieties is well-defined.

We need to show that the above map
\[(x,\,y,\,z)\,\longmapsto\, (\wt x,\,\wt y,\,\wt z)\,=\,(2-x^2,\,2-y^2,\,2-z^2)\]
is compatible with the birational map between the moduli space. 
Consider  an element  $[\nabla]\in\mathcal M_{1,1}^\rho$ determined by the monodromy representation $\Theta$. 
Let $X, Y \in \SL(2, \C)$ be the monodromies  along the loops  $\gamma_x$ and $\gamma_y$ on the one-punctured torus. 

Recall that the monodromy of  the connection $D^s$ on ${\mathcal O}_{\Sigma_2}(-2p_1)$ is given by $-1$ around the singularities
$p_l$. The generators $2,\,2\sqrt{-1}\,\in\, 2\Z+2\sqrt{-1}2\Z$  of the
lattice defining the torus $\Sigma_2$ define two generators of  the fundamental group of  $\Sigma_2$. 
The monodromy of  $D^s$ along these two generators of  $2\Z+2\sqrt{-1}2\Z$  is also $-1$.

Let $\Theta' $ be the monodromy representation
corresponding to the image in $M_{0,4}^{\wt\rho}$ of $[\nabla]\in\mathcal M_{1,1}^\rho$ through the birational map in the statement of the Theorem.  Denote by  $M_l=\Theta'(\gamma_l),$ $l=1,\dots,4$
  the  (local) monodromies of $\Theta'$  along  the simple oriented  loops $\gamma_{l}$ on $S_4$  going 
around the 4 punctures $ x_l\,:=\,e^{(l-1)\sqrt{-1}\frac{\pi }{2}} \in \CP^1$.

Consider the loops  $\gamma_{2}   \gamma_{1}$,  $\gamma_{3}   \gamma_{2}$ and  $\gamma_{3}   \gamma_{1}$  on $S_4$. Their images through the monodromy homomorphism $\Theta'$  are  $M_2M_1$, $M_3M_2$ and $M_3M_1.$ Lifting these curves to the four-punctured torus $\Sigma_2$
together with the above  properties of the monodromy of $D^s$ shows
\begin{equation}
\begin{split}
&M_2M_1\,\equiv\, -X^2\\
&M_3M_2\,\equiv\, -Y^2\\
&M_3M_1\,\equiv\, -(YX)^2,\\
\end{split}
\end{equation}
where $\equiv$ is the equivalence relation of lying
in the same conjugacy class. Taking traces yields the claimed map between the character varieties. Moreover, the local monodromies around the $4$ singular points $p_l$  in  $\Sigma_2$ are given by 
\[
M_l^2\,\equiv\,-Y^{-1}X^{-1}YX\]
Taking the trace gives $2-(2\cos(2\pi\wt\rho))^2\,= \,2\cos(2\pi\rho)$
corresponding to $\rho\,=\,2\wt \rho-\tfrac{1}{2}.$
\end{proof}

\begin{lemma}\label{abelred}
Let $D$ and $\wt\rho$ be as in Proposition \ref{red}. Then,
\[f_2^*D\otimes (D^s)^*\]
is given by $\pi_4^*\nabla^{a^0,\chi^0,\, \rho}$ with 
$\chi^0\,=\,\pi\tfrac{1-\sqrt{-1}}{4}$, $a^0\,=\,\pi\tfrac{1+\sqrt{-1}}{4}$ and $\rho\,=\, 2\wt\rho-\tfrac{1}{2}.$
\end{lemma}

\begin{proof}
This assertion follows from the proof of \cite[Theorem 3.2]{HHSch}. In the geometric context of \cite{HHSch} the 
connection $\nabla^{a^0,\chi^0,\, \rho}$ with $\chi^0\,=\,\pi\tfrac{1-\sqrt{-1}}{4}$ and 
$a^0\,=\,\pi\tfrac{1+\sqrt{-1}}{4}$ solves the extrinsic closing condition of a compact CMC surface in
the 3-sphere $\mathbb S^3$. Particular instances of minimal surfaces are given by the famous Lawson surfaces \cite{L}.
\end{proof}

\begin{lemma}\label{lem:FFuchs}
There exists a flat $\SL(2,\R)$-connection $\nabla^F$  in $\mathcal M_{1,1}^\rho$  such that $\pi_4^*\nabla^F$ and
$f_2^*\wt\nabla\otimes (D^s)^*$ are gauge equivalent on the four-punctured torus $\Sigma_2$.  The connections $\wt\nabla$
and $D^s$ are defined  in \eqref{wtnabla} and \eqref{Ds2} respectively.
\end{lemma}

\begin{proof}
Let $k\,\in\,\N^{\geq3}$, $\wt\rho\,=\,\frac{k-1}{2k}$, and consider the associated connection
$\wt\nabla$ in Lemma \ref{lem:realSU4}.
Using Lemma \ref{lem:realSU3}, the monodromy representation for $\wt\nabla$ 
is determined by the following characters 
\begin{equation}
\begin{split}
\wt x&\,=\,-2-4\cos{\frac{\pi}{k}}\\
\wt y&\,=\,-2-4\cos{\frac{\pi}{k}}\\
\wt z&\,=\,-2(2+4\cos{\frac{\pi}{k}}+\cos{\frac{2\pi}{k}})
\end{split}
\end{equation}
with $\mu\,=\,2\cos{\pi\frac{k-1}{k}}.$

Consider  the flat connection $\nabla^F$ on the one-punctured torus 
determined by the following  element of the character variety of the one-punctured torus 
\begin{equation}
\begin{split}
x&\,=\,2\sqrt{1+\cos{\frac{\pi}{k}}}\\
y&\,=\,2\sqrt{1+\cos{\frac{\pi}{k}}}\\
z&\,=\,4(\cos{\frac{\pi}{2k}})^2\\
\kappa&\,=\,2\cos{2\pi\frac{k-2}{2k}}\,=\,-2\cos{\frac{2\pi}{2k}}.
\end{split}
\end{equation}
Here $\rho\,=\,2\wt\rho-\tfrac{1}{2}\,=\, \frac{k-2}{2k}$ and $\nabla^F \in \mathcal M_{1,1}^\rho$. The proof of Theorem \ref{thm:4:1} shows that  $\pi_4^*\nabla^F$ and
$f_2^*\wt\nabla\otimes (D^s)^*$ define the same element in the character variety of the four-punctured torus $\Sigma_2$.
This implies that  $\pi_4^*\nabla^F$ and
$f_2^*\wt\nabla\otimes (D^s)^*$ are  gauge equivalent on the four-punctured torus $\Sigma_2$. 
\end{proof}

\begin{remark}
The connection $\wt\nabla$ in \eqref{wtnabla} does not admit a strongly  parabolic Higgs field with non-zero determinant, and the 
abelianization-procedure  does not apply  directly.  But Lemma \ref{lem:FFuchs} shows that it is possible to determine a connection $\nabla^F$ on $T^2$
 such that  $\pi_4^*\nabla^F$ and
$f_2^*\wt\nabla\otimes (D^s)^*$ are  gauge equivalent on the four-punctured torus $\Sigma_2$.  In \cite[Theorem 3.5]{HeHe}, the 
connection $\pi_4^*\nabla^F$ is written as a limit of connections of the form in \eqref{abel-connection}.
It can be shown that the underlying holomorphic bundle of  $\nabla^F$ is a non-trivial extension of the spin bundle by itself.
\end{remark}

\section{Proofs}

\subsection*{Proof of Theorem \ref{thi}}$\;$\\
Let $k\,\in\,\N^{\geq3}$, $\wt\rho\,=\,\frac{k-1}{2k}$ and $\rho\,=\, 2\wt\rho-\tfrac{1}{2}\,=\, \frac{k-2}{2k}$. 
Consider a sequence of distinct connections $\nabla^{t_n}$ with real monodromy, as constructed in Corollary \ref{tn}.
By Theorem \ref{thm:4:1}, the connection $\nabla^{t_n}$ induces a logarithmic connection on $\mathcal S_4$ with
real monodromy. This connection is given by 
\[D^{\tau_n}\,:=\,D+\tau_n \Phi,\]
where $\Phi$ is the strongly parabolic Higgs field for $D$ given in \eqref{Phi}, 
and
$\tau_n\,\in\,\C\setminus\{0\}$ is determined by $t_n.$ To be more explicit, 
the holomorphic quadratic differential $\det(\Phi)\,=\,\tfrac{4\sqrt{-1} (dz^2)}{z^4-1}$
pulls back to $c^2(dw)^2$ on $\Sigma_2$ for some $c\in\R^{>0}.$ Also
note that $D+h\Psi$ is gauge equivalent to $D-h\Psi$ for every $h\,\in\,\C.$
Thus, we have
\[\tau_n\,=\,\pi\frac{1+\sqrt{-1}}{4 \, c} t_n.\]

By Proposition \ref{red}, the pull-back of $D$ to $\Sigma_k$, through the map $f_k$ in (\ref{fk}), is gauge equivalent to the de Rham differential. The 
same gauge transformation sends $\Phi$ to a holomorphic Higgs field with respect to the trivial holomorphic 
structure by Lemma \ref{strong-pb}. Since $\nabla^{t_n}$ and $\nabla^F$ (constructed in 
Lemma \ref{lem:FFuchs}) are both SL$(2, \R)$-connections on the one-punctured torus $T^2\setminus\{o\}$, the map 
given in Theorem \ref{thm:4:1} sends them into the same real component of connections on $\mathcal S_4$ (that of 
${\wt\nabla}$ in (\ref{wtnabla})); see Remark \ref{4realcomponents}. By Proposition \ref{FRS} and Proposition 
\ref{component-pull-back} we obtain that the pull back of $D+\tau_n \Phi$ to $\Sigma_k$ is in the connected 
component with maximal Euler class $g-1\,=\,k-2.$
\qed

\subsection*{Proof of Corollary \ref{cor1}}$\,$\\
Consider for $\rho=\tfrac{k-2}{2k}$ 
the connections \[\nabla^t_\chi=\,\nabla^{ (1-t)\, a,\chi,\,\rho}\]
with
\[\chi\in (1\mp\sqrt{-1})\R\setminus\tfrac{1}{2}\Gamma^* \quad\text{and}\quad a \in (1\pm\sqrt{-1})\R\]
such that $\nabla^t$ is equivariant under the real involution  $\eta$, see Lemma \ref{tausymcon}.
Recall that the moduli space of S-equivalence classes of rank two stable bundles with trivial determinant over $\Sigma_k$ is a projective 
variety.
The subspace of (semistable) equivariant holomorphic bundles over $\Sigma_k$ identifies with
 the moduli space of corresponding parabolic structures on $\mathcal S_4$ by  pull-back and desingularization. As such it
is a projective line as explained in Section \ref{sssAbel}, see also \cite{LoSa}.
The two lines $\chi\in (1\mp\sqrt{-1})\R\setminus\tfrac{1}{2}\Gamma^*$ in the Jacobian
are mapped onto two  semicircles constituting a  circle in the aforementioned projective line. The trivial holomorphic structure
corresponding to 
$\chi^0$ is the only point contained in the intersection of the semicircles (as the holomorphic line bundles  determined by $\overline{\chi^0}$ and $\chi^0$ only differ by a spin bundle on $T^2$).
We refer to these as the {\em compatible real holomorphic structures} on $\Sigma_k.$
The only missing point in the circle is given by a wobbly bundle, where our method  does not apply.
The proof of Corollary \ref{cor1} works verbatim using $\nabla^t_\chi$ instead of $\nabla^t$.

\subsection*{Proof of Corollary \ref{cori}}$\;$\\
By Theorem \ref{thi} there exists a compact curve $\Sigma_3$ of genus $g\,=\,2$ and a holomorphic
connection 
$\nabla({\Sigma_3})$ on the rank two trivial holomorphic bundle over $\Sigma_3$ such that the monodromy homomorphism 
of $\nabla({\Sigma_3})$ is Fuchsian.

Consider an open neighborhood $\mathcal V$ of the monodromy of $\nabla({\Sigma_3})$ in the space of conjugacy 
classes of group homomorphisms $\pi_1(\Sigma_3) \,\longrightarrow\, \text{SL}(2, {\mathbb C})$ formed by 
quasi-Fuchsian representations. Recall that quasi-Fuchsian representations are faithful and their image in 
$\text{SL}(2, {\mathbb C})$ is a discrete group whose canonical action on $\CP^1$ has a Jordan curve as limit set 
and preserves each component of the domain of discontinuity. By Bers' simultaneous uniformization each conjugacy 
class of a quasi-Fuchsian representation is determined by the pair of elements in the Teichm\"uller space given 
by the quotient of the two connected components of the discontinuity domain by the image of the representation.

The main result in \cite{CDHL} gives an open neighborhood $\mathcal W$ of $(\Sigma_3,\, \nabla({\Sigma_3}))$ 
in the space of irreducible holomorphic differential systems (i.e., pairs of the form $(\Sigma,\, \nabla)$ where 
$\Sigma$ is an element in the Teichm\"uller space of compact curves of genus $g\,=\,2$ and $\nabla$ is an 
irreducible holomorphic $\text{SL}(2, {\mathbb C})$--connections on ${\mathcal O}^{\oplus 2}_\Sigma$) such that 
the restriction of the Hilbert-Riemann monodromy mapping to $\mathcal W$ is a biholomorphism between $\mathcal W$ 
and $\mathcal V$. This proves the first statement in the Corollary.

Consider now the open set $\mathcal U$ in the Teichm\"uller space of compact curves of genus $g\,=\,2$ which is 
the image of $\mathcal W$ through the natural forgetful projection. Take $\Sigma \in \mathcal U$ and 
$\nabla({\Sigma})$ a holomorphic connection on rank two holomorphic trivial  bundle  ${\mathcal O}^{\oplus 2}_{\Sigma}$
with quasi-Fuchsian monodromy 
representation. Denote by $\Gamma \subset \text{SL}(2, {\mathbb C})$ the image of the monodromy homomorphism for 
$\nabla({\Sigma})$.

Let $\widetilde{\Sigma}\,\longrightarrow\, {\Sigma}$ be the   universal cover of $\Sigma$ and let 
$\widetilde{\nabla}(\widetilde{\Sigma})$ be the pull-back of $\nabla({\Sigma})$ to the rank two  trivial holomorphic bundle  ${\mathcal O}^{\oplus 2}_{\widetilde{\Sigma}}$
 over 
$\widetilde{\Sigma}$ through the covering map.

Since 
$\widetilde{\nabla}(\widetilde{\Sigma})$ is flat and $\widetilde{\Sigma}$ is simply connected, there exists a 
global $\widetilde{\nabla}(\widetilde{\Sigma})$-parallel frame of the rank two trivial  bundle over $\widetilde{\Sigma}$. Such a parallel frame on the holomorphically trivial bundle ${\mathcal O}^{\oplus 2}_{\widetilde{\Sigma}}$  is determined by a holomorphic map 
${\widetilde{\Sigma}} \longrightarrow \text{SL}(2, {\mathbb C})$ which is equivariant with respect to two  actions of the 
fundamental group of $\Sigma$, namely by  deck transformations on $\widetilde{\Sigma} $ and through the monodromy 
morphism of  $\nabla({\Sigma})$ on $ \text{SL}(2, {\mathbb C})$. This provides a holomorphic map $\Sigma 
\longrightarrow \text{SL}(2, {\mathbb C}) / \Gamma$, with $\Gamma$ being the image of the monodromy homomorphism for 
$\nabla({\Sigma})$.
Here, we make   use of the holomorphic   trivialization of   ${\mathcal O}^{\oplus 2}_{\widetilde{\Sigma}}$ which is the pull-back of  the holomorphic   trivialization of  ${\mathcal O}^{\oplus 2}_{\Sigma}$.
 
Since $\nabla({\Sigma})$ is irreducible (and therefore nontrivial), the above  map is non-constant.
Notice that, up to a finite index subgroup (and an associated finite cover of the target), we can assume that 
$\Gamma$ is torsion free and hence $\text{SL}(2, {\mathbb C})/ \Gamma$ is a complex threefold (without orbifold 
points).

Moreover, such quotients of $\text{SL}(2, {\mathbb C})$ are diffeomorphic to the orthonormal frame bundle of the 
associated quasi-Fuchsian hyperbolic 3-manifold 
 (which is known to be isometric to the quotient of a convex set in 
the hyperbolic 3-space by the quasi-Fuchsian group of hyperbolic isometries). Note that the boundary of the 
quasi-Fuchsian manifold has two connected components that are conformally equivalent to the pair  of 
points in the Teichm\"uller space given by Bers' simultaneous uniformization; the complex structure on the 
oriented orthonormal frame bundle of the quasi-Fuchsian manifold comes from the identification of the orientation 
preserving isometry group $\text{PSL}(2, {\mathbb C})$ with the oriented orthonormal frame bundle of the 
hyperbolic 3-space \cite{Gh}.
\qed

We would like to formulate  a general problem  similar  to that of Ghys  and  to the questions asked  in \cite{CDHL, Ka}. Consider a compact orientable surface $S_g$ of genus $g \geq 2$.  Characterize the conjugacy classes of 
$\text{SL}(2, {\mathbb C})$-representations of the fundamental group of $S_g$ such that the associated rank two 
flat vector bundle over $S_g$ is holomorphically trivial with respect to some point in the Teichm\"uller space of 
$S_g$.

The analogous question for the uniformization bundle has been answered completely in \cite{GaKaMa}.
Note that a holomorphic $\text{SL}(2, {\mathbb C})$-connection 
on the uniformization bundle gives rise to a complex projective structure on the Riemann surface and vice versa
after the choice of a theta characteristic.
For the case of the trivial rank one bundle  this question  was answered in \cite{Ha} (see also \cite{Ka} where this result was rediscovered).

\newpage

\appendix

\section{A result on WKB approximation}\label{A1}

\medskip
\begin{center}
By {\bf Takuro Mochizuki}
\end{center}
\medskip

\subsection{Limiting behavior of a family of flat connections}\label{subsection;21.3.7.30}

Let $X$ be a Riemann surface,
which is not necessarily compact.
Let $V$ be a vector bundle on $X$
equipped with a flat $\SL_2(\cnum)$-connection $\nabla$.
Let $\delbar_V$ denote the induced holomorphic structure of $V$.
Let $\Phi$ be a Higgs field of the holomorphic vector bundle
$(V,\delbar_V)$ such that $\tr\Phi=0$.
We obtain the family of flat connections $\nabla^t=\nabla+t\Phi$
on $V$ $(t\geq 0)$.
\begin{assumption}
We assume that
there exist a holomorphic one form $\omega$
and a decomposition $V=V_+\oplus V_-$
such that $\Phi=\omega(\pi_{V_+}-\pi_{V_-})$,
where $\pi_{V_{\pm}}$ denote the projections of
$V$ onto $V_{\pm}$ with respect to the decomposition.
\hfill\qed
\end{assumption}

Note that there exists a unique decomposition
$\nabla=\nabla^{\circ}+f$,
where $\nabla^{\circ}$ is
the direct sum of connections $\nabla_{V_{\pm}}$ of
$V_{\pm}$,
and $f$ is a holomorphic section of
$\Bigl(
\Hom(V_1,V_2)
\oplus
\Hom(V_2,V_1)
\Bigr)\otimes\Omega^1$.

We set $[0,1]:=\{0\leq u\leq 1\}$.
Let $\gamma:[0,1]\lrarr X$ be a
$C^{\infty}$-path
which is a WKB-curve with respect to $\omega$,
i.e.,
\[
\Re\bigl(\gamma^{\ast}(\omega)(\del_u)\bigr)<0
\]
at any point of $[0,1]$.
Let $\nbigp^t_{\gamma}:
V_{|\gamma(0)}\simeq V_{|\gamma(1)}$
denote the isomorphism
obtained as the parallel transport of
$\nabla+t\Phi$ along $\gamma$.
Similarly,
let $\nbigp_{\pm,\gamma}$ denote the isomorphisms
$V_{\pm|\gamma(0)}\simeq
V_{\pm|\gamma(1)}$
obtained as the parallel transport of $\nabla_{V_{\pm}}$
along $\gamma$.

We shall explain a proof of the following proposition in
\S\ref{subsection;21.3.7.20}
after preliminaries in
\S\ref{subsection;21.3.7.21}--\S\ref{subsection;21.3.7.22}.
\begin{proposition}
\label{prop;21.3.5.20}
 For $(w_+,w_-)\in V_{|\gamma(0)}=V_{+|\gamma(0)}\oplus V_{-|\gamma(0)}$,
we have
\[
 \lim_{t\to\infty}
 e^{t\int_{\gamma}\omega}
 \cdot \nbigp^t_{\gamma}(w_+,w_-)
=\bigl(
 \nbigp_{+,\gamma}(w_+),0
 \bigr)
\in V_{+|\gamma(1)}\oplus V_{-|\gamma(1)}.
\] 
\end{proposition}

\begin{remark}
Proposition {\rm\ref{prop;21.3.5.20}}
and its proof are essentially explained in
{\rm\cite[Appendix C]{GMN}}.
Hopefully,
a more detailed explanation in this appendix would be useful.
It is also closely related to
the Riemann-Hilbert WKB problem studied in
{\rm \cite{Katzarkov-Noll-Pandit-Simpson}}.
\hfill\qed
\end{remark}

We obtain the following corollary
as an immediate consequence of
Proposition \ref{prop;21.3.5.20}.

\begin{corollary}
\label{cor;21.3.7.10}
If $\gamma$ is closed, i.e., $\gamma(0)=\gamma(1)$,
we obtain
\[
 \lim_{t\to\infty}\tr(\nbigp^t_{\gamma})e^{t\int_{\gamma}\omega}
 =\tr(\nbigp_{+,\gamma})\neq 0.
\]
 \hfill\qed
\end{corollary}

\subsection{An elementary lemma}
\label{subsection;21.3.7.21}

For any $s_1<s_2$,
we set $[s_1,s_2]:=\{s_1\leq s\leq s_2\}$.
For any non-negative integer $\ell$,
let $C^{\ell}([s_1,s_2])$ denote the space of
$\cnum$-valued $C^{\ell}$-functions on $[s_1,s_2]$.
For any $f\in C^0([s_1,s_2])$,
we set $\|f\|_{C^0([s_1,s_2])}:=\max_{s\in [s_1,s_2]}|f(s)|$.

Fix $\rho>0$, $\epsilon>0$ and $C_0>0$.
Suppose that
$\alpha\in C^0([0,1])$ satisfies
$\Re(\alpha(s))>\rho$ for any $s\in [0,1]$.
Let $\beta\in C^0([0,1])$
such that $\|\beta\|_{C^0([0,1])}\leq C_0$.
Suppose that $f^t\in C^1([0,1])$ $(t\geq 0)$ satisfies
\[
 \|f^t\|_{C^0([0,1])}
 +\|\del_sf^t+(t\alpha+\beta)f^t\|_{C^0([0,1])}
 \leq \epsilon.
\]
Take $0<\delta<1$.
We recall the following standard and elementary lemma,
which we prove just for the convenience of the reader.

\begin{lemma}
\label{lem;21.3.5.3}
 There exist $C_1>0$ and $t_1>0$,
depending only on $C_0$, $\rho$ and $\delta$
 such that the following holds
 for any $t\geq t_1$:
\[
 \|f^t\|_{C^0([\delta,1])}
\leq C_1\epsilon(1+t)^{-1}
\]
\end{lemma}

\begin{proof}
We set $\alphatilde(s):=\int_0^s\alpha(u)\,du$
and $\betatilde(s):=\int_0^s\beta(u)\,du$.
We have
$\|\betatilde\|_{C^0([0,1])}\leq C_0$.
For any $0\leq s_1\leq s_2$,
we have
\[
\Re\alphatilde(s_2)
-\Re\alphatilde(s_1)
=\int_{s_1}^{s_2}\Re\alpha(u)\,du
>
\rho (s_2-s_1).
\]

We set
$g^t:=\del_sf^t+(t \alpha+\beta)f^t$.
Because
$\del_s\bigl(
e^{t\alphatilde+\betatilde}f^t
\bigr)
=e^{t\alphatilde+\betatilde}g^t$,
we obtain
\[
f^t=e^{-t\alphatilde(s)-\betatilde(s)}
\int_0^s
e^{t\alphatilde(u)+\betatilde(u)}g^t(u)\,du
+e^{-t\alphatilde(s)-\betatilde(s)}f^t(0).
\]
We have
$\|
e^{-t\alphatilde(s)-\betatilde(s)}f^t(0)
\|_{C^{0}([\delta,1])}
\leq \epsilon e^{-t\rho\delta+C_0}$.
We also have the following inequalities
for $s\in [0,1]$:
\[
\left|
e^{-t\alphatilde(s)-\betatilde(s)}
\int_0^s
e^{t\alphatilde(u)+\betatilde(u)}g^t(u)\,du
\right|
\leq
 \int_0^s
 e^{-t\rho(s-u)+2C_0}\epsilon\,du
\leq \frac{\epsilon e^{2C_0}}{\rho t}.
\]
Then, we obtain the claim of the lemma.
\end{proof}

\vspace{.1in}
Let us state a variant.
Suppose that
$\alpha_1\in C^0([0,1])$ satisfies
$\Re(\alpha_1(s))<-\rho$
for any $s\in[0,1]$.
Let $\beta_1\in C^0([0,1])$
such that $\|\beta_1\|_{C^0([0,1])}\leq C_0$.
Suppose that
$f^t_1\in C^1([0,1])$ $(t\geq 0)$ satisfies
\[
 \|f^t_1\|_{C^0([0,1])}
 +\|\del_sf_1^t+(t\alpha_1+\beta_1)f^t_1\|_{C^0([0,1])}
 \leq \epsilon.
\]

\begin{lemma}
\label{lem;21.3.5.4}
The following inequality holds for any $t\geq t_1$:
\[
 \|f^t_1\|_{C^0([0,1-\delta])}
 \leq
 C_1\epsilon(1+t)^{-1}.
\]
Here, $C_1$ and $t_1$ are positive constants
in Lemma {\rm\ref{lem;21.3.5.3}}.
\end{lemma}

\begin{proof}
It is enough to apply Lemma \ref{lem;21.3.5.3}
to the function $f^t_1(1-s)$.
\end{proof}

\subsection{A singular perturbation theory}
\label{subsection;21.3.7.22}

We recall some results from \cite[\S2.4]{Decouple}
with a complementary estimate
for the convenience of the reader.

\subsubsection{Notation}

Let $r$ be a positive integer.
Let $M_r(\cnum)$ denote the space of
$r\times r$ complex matrices.
Let $M_r(\cnum)_0\subset M_r(\cnum)$ denote
the subspace of diagonal matrices,
and let $M_r(\cnum)_1\subset M_r(\cnum)$
denote the subspace of off-diagonal matrices,
i.e.,
\[
 M_r(\cnum)_0=
 \bigl\{(a_{i,j})\in M_r(\cnum)\,\big|\,
a_{i,j}=0\,\,(i\neq j)
\bigr\},
\]
\[
M_r(\cnum)_1=
 \bigl\{(a_{i,j})\in M_r(\cnum)\,\big|\,
 a_{i,j}=0\,\,(i=j)
\bigr\}. 
\]

For any non-negative integer $\ell$,
let $C^{\ell}([s_1,s_2],M_r(\cnum))$ denote the space of
$M_r(\cnum)$-valued $C^{\ell}$-functions on $[s_1,s_2]$.
Similarly,
let $C^{\ell}([s_1,s_2],M_r(\cnum)_{\kappa})$ $(\kappa=0,1)$
denote the spaces of
$M_r(\cnum)_{\kappa}$-valued $C^{\ell}$-functions on $[s_1,s_2]$.
We set
$\|Y\|_{C^0([s_1,s_2])}:=\max_{i,j}\|Y_{i,j}\|_{C^0([s_1,s_2])}$
for any $Y\in C^0([s_1,s_2],M_r(\cnum))$.

\subsubsection{Gauge transformations}
\label{subsection;21.3.7.1}

Fix $C_0>0$.
Suppose that $a_j,b_j\in C^0([0,1])$ $(j=1,\ldots,r)$
satisfy the following conditions.
\begin{itemize}
 \item $\Re a_1(s)<\Re a_2(s)<\cdots<\Re a_r(s)$ for any $s\in[0,1]$.
 \item $\|b_j\|_{C^0([0,1])}\leq C_0$. 
\end{itemize}

For $t\geq 0$,
let $A^t(s)$ denote the $M_r(\cnum)_0$-valued function
whose $(i,i)$-entries are $ta_i(s)+b_i(s)$.
The following proposition is proved in
\cite[Proposition 2.18]{Decouple}.

\begin{proposition}
\label{prop;21.3.5.1}
There exist $C_1>0$ and $\epsilon_1>0$,
depending only on $C_0$,
 such that the following holds:
\begin{itemize}
 \item For any $t\geq 0$ and any $B\in C^0([0,1],M_r(\cnum)_1)$ satisfying
       $\|B\|_{C^0([0,1])}\leq \epsilon_1$,
       there exist $G^t\in C^1([0,1],M_r(\cnum)_1)$
       and $H^t\in C^0([0,1],M_r(\cnum)_0)$
       satisfying
\begin{equation}
       \|G^t\|_{C^0([0,1])}+\|\del_sG^t+[A^t,G^t]\|_{C^0([0,1])}
       +\|H^t\|_{C^0([0,1])}\leq C_1\|B\|_{C^0([0,1])},
\end{equation}
\begin{equation}
\label{eq;21.3.7.2}
       A^t+B=(I+G^t)^{-1}(A^t+H^t)(I+G_t)
       +(I+G^t)^{-1}\del_sG^t. 
\end{equation}
Here, $I\in M_r(\cnum)$ denote the identity matrix.
\hfill\qed
\end{itemize} 
\end{proposition}

\begin{remark}
In Proposition {\rm\ref{prop;21.3.5.1}},
we assume that $C_1\epsilon_1$ is sufficiently small
so that $I+G^t$ is invertible.
\hfill\qed 
\end{remark}

Let us add a complementary estimate to Proposition \ref{prop;21.3.5.1}.
There exist $C_2>0$ and $C_3>0$ such that
(i) $\Re(a_{i+1}(s)-a_i(s))>C_2$
for any $s\in[0,1]$ and $i=1,\ldots,r-1$,
(ii) $\|a_i\|_{C^0([0,1])}\leq C_3$ for any $i$.
Take $0\,<\,\delta\,<\,\tfrac{1}{2}$.

\begin{lemma}
\label{lem;21.3.5.11}
 There exist  $C_4\,>\,0$ and $t_4\,>\,0$,
depending only on $C_0$, $C_2$, $C_3$ and $\delta$,
such that the following holds on $[\delta,1-\delta]$
for $t\geq t_4$:
\begin{itemize}
 \item Let $G^t$ and $H^t$ be as in Proposition {\rm\ref{prop;21.3.5.1}}.
       Then, we have
\[
       \bigl\|G^t\bigr\|_{C^0([\delta,1-\delta])}
       +\bigl\|H^t\bigr\|_{C^0([\delta,1-\delta])}
       \leq C_4 (1+t)^{-1}\|B\|_{C^0([0,1])}.
\]
\end{itemize} 
\end{lemma}

\begin{proof}
Note that $G^t_{i,i}=0$ for any $i$.
For $i\neq j$,
we have
\begin{equation}
\label{eq;21.3.5.10}
\bigl\|G^t_{i,j}\bigr\|_{C^0([0,1])}+
\bigl\|
 \del_sG^t_{i,j}
+\bigl(t(a_i-a_j)+b_i-b_j\bigr)G^t_{i,j}
\bigr\|_{C^0([0,1])}
\leq C_1\|B\|_{C^0([0,1])}.
\end{equation}
By Lemma \ref{lem;21.3.5.3} and Lemma \ref{lem;21.3.5.4},
there exist $C_{10}>0$ and $t_{10}>0$,
depending only on $C_0$, $C_2$ and $\delta$
such that the following holds for any $t\geq t_{10}$:
\begin{equation}
\label{eq;21.3.5.11}
 \|G^t\|_{C^0([\delta,1-\delta])}
\leq \frac{C_{10}}{1+t}\|B\|_{C^0([0,1])}.
\end{equation}
By (\ref{eq;21.3.5.10}) and (\ref{eq;21.3.5.11}),
there exist $C_{11}>0$,
depending only on $C_0$, $C_2$, $C_3$ and $\delta$
such that the following holds for any $t\geq t_{10}$:
\begin{equation}
\label{eq;21.3.5.12}
 \|\del_sG^t\|_{C^0([\delta,1-\delta])}
\leq C_{11}\|B\|_{C^0([0,1])}.
\end{equation}
By (\ref{eq;21.3.7.2}),
we have 
\begin{equation}
\label{eq;21.3.5.2}
 (I+G^t)(A^t+B)(I+G^t)^{-1}
=A^t+H^t+\del_s(G^t)\cdot (I+G_t)^{-1}.
\end{equation}
Note that the diagonal entries of 
$B$, $G$ and $\del_sG^t$ are $0$.
By (\ref{eq;21.3.5.11}), (\ref{eq;21.3.5.12})
and (\ref{eq;21.3.5.2}),
there exist $C_{12}>0$,
depending only on $C_0$, $C_2$, $C_3$ and $\delta$
such that the following holds for any $t\geq t_{10}$:
\[
 \|H^t\|_{C^0([\delta,1-\delta])}
 \leq C_{12}(1+t)^{-1}\|B\|_{C^0([0,1])}.
\]
Thus, we obtain the claim of the lemma.
\end{proof}

\subsubsection{Reformulation}

Let us recall the reformulation of
Proposition \ref{prop;21.3.5.1}
with a complementary estimate,
as in \cite[Corollary 2.19]{Decouple}.
Let $A^t$, $C_i$ $(i=0,1,2,3)$ and $\epsilon_1$ be as in
\S\ref{subsection;21.3.7.1}.
Let $E$ be a $C^1$-vector bundle on $[0,1]$
with a frame $\vecv=(v_1,\ldots,v_r)$.
 Let $B\in C^0([0,1],M_r(\cnum)_1)$
 satisfying
 $\|B\|_{C^0([0,1])}\leq \epsilon_1$.
For $t\geq 0$,
let $\nabla^t$ denote the connection of $E$
determined by
$\nabla^t\vecv=\vecv\cdot (A^t+B)\,ds$.
We obtain the following corollary
from Proposition \ref{prop;21.3.5.1}
and Lemma \ref{lem;21.3.5.11}.

\begin{corollary}
\label{cor;21.3.5.12}
There exist matrix valued functions
$G^t\in C^1([0,1],M_r(\cnum)_1)$
and $H^t\in C^0([0,1],M_r(\cnum)_0)$
such that the following holds.
\begin{itemize}
 \item
      $\|G^t\|_{C^0([0,1])}+\|\del_sG^t+[A^t,G^t]\|_{C^0([0,1])}
   +\|H^t\|_{C^0([0,1])}\leq C_1\|B\|_{C^0([0,1])}$.
 \item
       For the frame
 $\vecu^t=\vecv\cdot (I+G^t)^{-1}$,
 we have
$\nabla^t\vecu^t=\vecu^t\cdot (A^t+H^t)\,ds$.
\end{itemize}
Moreover, there exist $C_{20}>0$ and $t_{20}>0$
depending only on $C_i$ $(i=0,2,3)$
such that the following holds for any $t\geq t_{20}$:
\[
 \|G^t\|_{C^0([1/4,3/4])}
 +\|H^t\|_{C^0([1/4,3/4])}
 \leq \frac{C_{20}}{1+t}\|B\|_{C^0([0,1])}
\] 
\hfill\qed
\end{corollary}

For each $t$,
$\nabla^t$ induces an isomorphism
$\Psi^t:E_{|s=1/4}\simeq E_{|s=3/4}$.
It is represented by
the diagonal matrix with respect to
the bases
$\vecu^t_{|s=1/4}$ and $\vecu^t_{|s=3/4}$,
whose $(j,j)$-entries are
\[
 \exp\Bigl(
 -\int_{1/4}^{3/4}\bigl(ta_j(s)+b_j(s)+H_{jj}^t(s)\bigr)\,ds
 \Bigr).
\]

\subsection{Proof of Proposition \ref{prop;21.3.5.20}}
\label{subsection;21.3.7.20}

Let us return to the setting in \S\ref{subsection;21.3.7.30}.
We extend $\gamma$ to a $C^{\infty}$-map
$\gammatilde:[-1,2]\lrarr X$
such that
$\Re\gammatilde^{\ast}\omega(\del_u)<0$
at any point of $[-1,2]$.
There exists a $C^{\infty}$-frame $v_{\pm}$
of $\gammatilde^{\ast}V_{\pm}$.
We have
$\gammatilde^{\ast}(\Phi)(v_{\pm})
=\pm\gammatilde^{\ast}(\omega) v_{\pm}$.
We obtain a $C^{\infty}$-map
$\nbigb:[-1,2]\lrarr M_2(\cnum)$
determined by
\[
\gammatilde^{\ast}(\nabla)(v_+,v_-)
=(v_+,v_-)\cdot\nbigb\,du.
\]
We have
$\gammatilde^{\ast}(\nabla_{V_{+}})v_+
=\nbigb_{11}v_+\,du$
and
$\gammatilde^{\ast}(\nabla_{V_{-}})v_-
=\nbigb_{22}v_-\,du$.

We obtain $\alpha\in C^{\infty}([-1,2])$ by
$\gammatilde^{\ast}\omega=\alpha\,du$.
We have $\Re(\alpha)<0$
at any point of $[-1,2]$.
Let 
$\nbiga:[-1,2]\lrarr M_2(\cnum)_0$
be the $C^{\infty}$-map determined by
$\nbiga_{11}=\alpha$ and $\nbiga_{22}=-\alpha$.
We have
\[
 \gammatilde^{\ast}(\nabla^t)(v_+,v_-)
 =(v_+,v_-)\cdot\Bigl(
  t\nbiga+\nbigb
 \Bigr)\,du.
\]
There exists $C_0>0$ such that
$\|\nbigb_{j,j}\|_{C^0([-1,2])}\leq C_0$ for $j=1,2$.
Let $C_1$ and $\epsilon_1$ be positive constants
as in Proposition \ref{prop;21.3.5.1}, depending on $C_0$.
There exists a positive integer $N>10$
such that
\[
 \|\nbigb_{1,2}\|_{C^0([-1,2])}
 +\|\nbigb_{2,1}\|_{C^0([-1,2])}
 \leq \frac{N}{10}\epsilon_1.
\]
We set $u(i):=\frac{i}{N}$ for $i=-N,\ldots,2N$.
We obtain the decomposition
$[-1,2]=\bigcup_{i=-N}^{2N-1}[u(i),u(i+1)]$.
Let 
$\Pi^t_i:\gammatilde^{\ast}(V)_{|u(i)}
\simeq \gammatilde^{\ast}(V)_{|u(i+1)}$
denote the isomorphisms
obtained as the parallel transport of $\gammatilde^{\ast}\nabla^t$.

\begin{lemma}
\label{lem;21.3.5.12}
 There exist constants $C_{30}>0$ and $t_{30}>0$,
a family of
$2\times 2$-matrices
$G^{t}_{i,0},G^t_{i,1}
 \in M_2(\cnum)_1$
 for $t\geq 0$ and $-N\leq i\leq 2N-1$,
 and families of continuous functions
 $H^t_{i,+},H^t_{i,-}\in
  C^0([u(i),u(i+1)])$ for $-N\leq i\leq 2N-1$,
 such that the following holds.
\begin{itemize}
 \item $C_{30}(1+t_{30})^{-1}\leq 1/10$.
 \item $|G^t_{i,0}|+|G^t_{i,1}|\leq C_{30}(1+t)^{-1}$
       for any $t\geq t_{30}$.
       Note that we obtain the bases
       $(v_{+},v_-)_{|u(i)}(I+G^t_{i,0})^{-1}$
       and $(v_+,v_-)_{|u(i+1)}(I+G^t_{i,1})^{-1}$
       of
       $\gammatilde^{\ast}(V)_{|u(i)}$
       and
       $\gammatilde^{\ast}(V)_{|u(i+1)}$,
       respectively.
 \item $\|H^t_{i,\pm}\|_{C^0([u(i),u(i+1)])}\leq C_{30}(1+t)^{-1}$
       for any $t\geq t_{30}$.
 \item For each $(i,t)$, $\Pi^t_i$ is represented by
       a diagonal matrix $\nbigc^t_i$
       with respect to the bases
       $(v_{+},v_-)_{|u(i)}(I+G^t_{i,0})^{-1}$
       and $(v_+,v_-)_{|u(i+1)}(I+G^t_{i,1})^{-1}$.
       Moreover, we obtain
\[
       (\nbigc^t_i)_{1,1}
       =\exp\left(-\int_{u(i)}^{u(i+1)}
       (t\alpha
       +\nbigb_{1,1}+H^t_{i,+})\,du
       \right),
\]
\[
       (\nbigc^t_i)_{2,2}
       =\exp\left(-\int_{u(i)}^{u(i+1)}
       (-t\alpha
       +\nbigb_{2,2}+H^t_{i,-})\,du
       \right).       
\]       
\end{itemize}
\end{lemma}

\begin{proof}
Let $F_i:[0,1]\simeq [\frac{2i-1}{2N},\frac{2i+3}{2N}]$
be the affine isomorphism given by
$F_i(s)=\frac{1}{2N}(2i-1+4s)$.
Note that $F_i$ induces
$[\frac{1}{4},\frac{3}{4}]\simeq [u(i),u(i+1)]$.
We obtain the bundle
$F_i^{\ast}\bigl(\gammatilde^{\ast}V\bigr)$,
equipped with the frame $F_i^{\ast}(v_+,v_-)$
and the family of connections
$F_i^{\ast}\bigl(
\gammatilde^{\ast}(\nabla^t)
\bigr)$.
Because
$F_i^{\ast}(du)=\frac{2}{N}ds$,
we obtain
$\|F_i^{\ast}\nbigb_{j,j}\cdot \del_sF_i^{\ast}(u)
\|_{C^0([0,1])}\leq C_0$
and
\[
\bigl\|
 F_i^{\ast}\nbigb_{1,2}
 \cdot \del_sF_i^{\ast}(u)
 \bigr\|_{C^0([0,1])}
+ \bigl\|
  F_i^{\ast}\nbigb_{2,1}
  \cdot \del_sF_i^{\ast}(u)
 \bigr\|_{C^0([0,1])}
 \leq \epsilon_1.
\]
By applying Corollary \ref{cor;21.3.5.12}
to 
$F_i^{\ast}\bigl(\gammatilde^{\ast}V\bigr)$
with $F_i^{\ast}(v_+,v_-)$
and $F_i^{\ast}\bigl(
\gammatilde^{\ast}(\nabla^t)
\bigr)$,
we obtain Lemma \ref{lem;21.3.5.12}.
\end{proof}

\vspace{.1in}
We set
$\nbigctilde^t_i:=
(I+G^t_{i,1})^{-1}\nbigc^t_i
\cdot (I+G^t_{i,0})$.
Note that
$\Pi^t_i$ is represented by
$\nbigctilde^t_i$
with respect to the bases
$(v_+,v_-)_{|u(i)}$
and
$(v_+,v_-)_{|u(i+1)}$.
We set
\[
 \nbigd^t:=
 \nbigctilde^t_{N-1}\cdot
 \nbigctilde^t_{N-2}\cdot
 \cdots
 \cdot\nbigctilde^t_{1}\cdot
 \nbigctilde^t_0.
\]
Let $\Pi^t$ be the isomorphism
$\gammatilde^{\ast}(V)_{|0}
\simeq
\gammatilde^{\ast}(V)_{|1}$
obtained as the parallel transport of
$\gammatilde^{\ast}\nabla^t$.
Because
$\Pi^t=\Pi^t_{N-1}\circ \Pi^t_{N-2}\circ\cdots
\circ\Pi^t_1\circ\Pi^t_0$,
the isomorphism $\Pi^t$ is represented by
$\nbigd^t$
with respect to the bases
$(v_+,v_-)_{|0}$ and $(v_+,v_-)_{|1}$.
For any $1\leq k,\ell\leq 2$, we have
\[
 \lim_{t\to\infty}
 e^{t\int_{u(i)}^{u(i+1)}\gammatilde^{\ast}\omega}
 \cdot\nbigctilde^{t}_{k,\ell}
 =\left\{
\begin{array}{ll}
 \exp\bigl(-\int_{u(i)}^{u(i+1)}\nbigb_{1,1}\,du\bigr)
  & (\mbox{$(k,\ell)=(1,1)$})\\
 0 & (\mbox{otherwise}).
\end{array}
 \right.
\]
We obtain
\[
 \lim_{t\to\infty}
 e^{t\int_0^1\gammatilde^{\ast}\omega}
 \nbigd^{t}_{k,\ell}
 =\left\{
\begin{array}{ll}
 \exp\bigl(-\int_{0}^{1}\nbigb_{1,1}\,du\bigr)
  & (\mbox{$(k,\ell)=(1,1)$})\\
 0 & (\mbox{otherwise}).
\end{array}
 \right.
\]
Thus, we obtain the claim of Proposition \ref{prop;21.3.5.20}.
\hfill\qed

\section*{Acknowledgements}

IB is partially supported by a J. C. Bose Fellowship. SD was partially supported by the French government through the UCAJEDI Investments in the 
Future project managed by the National Research Agency (ANR) with the reference number 
ANR2152IDEX201. LH is supported by  the DFG grant HE
7914/2-1
of the DFG priority program SPP 2026 Geometry at Infinity.
SH is supported by the DFG grant HE 6829/3-1 of the DFG priority program SPP 2026 Geometry at Infinity.
TM is partially supported by
the Grant-in-Aid for Scientific Research (S) (No. 17H06127),
the Grant-in-Aid for Scientific Research (S) (No. 16H06335),
and the Grant-in-Aid for Scientific Research (C) (No. 20K03609),
Japan Society for the Promotion of Science. 
We would also like to thank Andrew Neitzke, Sebastian Schulz and Carlos Simpson for very helpful comments on WKB analysis. 

\end{document}